\documentclass[a4paper,reqno]{paper}

\usepackage[british]{babel}
\usepackage{mathptmx}
\DeclareMathAlphabet{\mathcal}{OMS}{cmsy}{m}{n}
\usepackage{courier}
\usepackage{amssymb}
\usepackage{hyperref}
\usepackage{mathtools}
\usepackage{amsthm}
\usepackage{enumerate}
\usepackage{enumitem}
\usepackage{tikz}
\usepackage{nicefrac}
\usepackage{extarrows}
\newcommand{\dotabovearrow}{
   \mathrel{\ooalign{\hss\raise.95ex\hbox{\scalebox{1.35}{\normalfont .}}%
   \kern0.35ex\hss\cr$\rightarrow$}}}
\newcommand{\dotbelowarrow}{
   \mathrel{\ooalign{\hss\raise-.07ex\hbox{\scalebox{1.35}{\normalfont .}}%
   \kern0.35ex\hss\cr$\rightarrow$}}}

\newtheorem{theorem}{Theorem}
\newtheorem{proposition}[theorem]{Proposition}
\newtheorem{corollary}[theorem]{Corollary}
\newtheorem{lemma}[theorem]{Lemma}
\newtheorem{definition}{Definition}

\newtheorem*{remark*}{Remark}

\newcommand{\nats}{\mathbb{N}}
\newcommand{\natswith}{\nats_{0}}
\newcommand{\reals}{\mathbb{R}}

\newcommand{\states}{\mathcal{X}}

\newcommand{\lt}{\underline{T}}
\newcommand{\ut}{\overline{T}}

\newcommand{\gambles}{\mathcal{L}}
\newcommand{\gamblesX}{\gambles(\states)} 

\newcommand{\ind}[1]{\mathbb{I}_{#1}}

\newcommand{\rateset}{\mathcal{Q}}
\newcommand{\lrate}{\underline{Q}}
\newcommand{\urate}{\overline{Q}}

\newcommand{\asa}{\Leftrightarrow}
\newcommand{\then}{\Rightarrow}

\newcommand{\norm}[1]{\left\lVert #1 \right\rVert}
\newcommand{\abs}[1]{\left\vert #1 \right\vert}

\begin{document}
\title{Convergence of Imprecise Continuous-Time Markov Chains}
\author{Jasper De Bock}

\maketitle

\begin{abstract}
We study the limit behaviour of a generally non-linear ordinary differential equation whose solution is a superadditive generalisation of a stochastic matrix, and provide necessary and sufficient conditions for this solution to be ergodic, in the sense that it converges to an operator that, essentially, maps functions to constants. In the linear case, the solution of our differential equation is equal to the matrix exponential of an intensity matrix and can then be interpreted as the transition operator of a homogeneous continuous-time Markov chain. 
Similarly, in the generalised non-linear case that we consider, the solution can be interpreted as the lower transition operator of a specific set of non-homogeneous continuous-time Markov chains, called an imprecise continuous-time Markov chain.
In this context, our main result provides a necessary and sufficient condition for such an imprecise continuous-time Markov chain to converge to a unique limiting distribution.
\end{abstract}


\section{Introduction}



Consider a real-valued $n\times n$ matrix $Q$ and let $T_t$ be a real-valued time-dependent  $n\times n$ matrix such that
\begin{equation*}
\frac{d}{dt}T_t=Q T_t \text{~~for all $t\geq0$}
\end{equation*}\\[-1pt]
and $T_0=I$, with $I$ the $n$-dimensional unit matrix.
The unique solution of this differential equation is then well known to be given by the matrix exponential $e^{Qt}$ of $Q$. If $Q$ is furthermore an intensity matrix---has non-negative off-diagonal elements and rows that sum to zero---then $T_t=e^{Qt}$ will be a stochastic matrix. 
In that case, $T_t$ can be interpreted as the transition operator of a homogeneous continuous-time Markov chain. Indeed, if we identify $\{1,\dots,n\}$ with the state space $\states$ of such a Markov chain and let $Q$ be its transition rate matrix, then for any two states $x,y\in\states$, $T_t(x,y)$ is the probability $P(X_t=y\vert X_0=x)$ of ending up in state $y$ at time $t$, conditional on starting in state $x$ at time zero. 

Rather remarkably, for any transition rate matrix $Q$, the conditional probability $P(X_t=y\vert X_0=x)$ will always converge~\cite[Theorem~II.10.1]{Chung:1967}. However, in general, this limiting value may depend on the initial state $x$. If this is not the case, that is, if there is a probability mass function $P_\infty$ on $\states$ such that
\begin{equation*}
\lim_{t\to+\infty}P(X_t=y\vert X_0=x)=P_\infty(y)
\text{~~for all $y\in\states$,}
\end{equation*}
then the homogeneous continuous-time Markov chain under consideration---or, equivalently, the transition rate matrix $Q$---is said to have a unique limiting distribution~$P_\infty$. From an applied point of view, the existence of such a limiting distribution is clearly important, because it implies that for large enough values of $t$, predicting the current value of $X_t$ does not require any knowledge about its initial values. Hence, we are led the following question: what conditions does $Q$ need to satisfy in order for $P_\infty$ to exist? As it turns out, this question has an elegant answer: the required conditions are relatively easy, and are fully determined by the signs of the components of $Q$; see for example Anderson~\cite{Anderson:1991}.

Our main goal here is to answer a non-linear generalisation of this question, which includes the linear version that we have just discussed as a special case. Basically, the only difference is that the intensity matrix $Q$ is replaced by a lower transition rate operator $\lrate$, which is a non-linear---superadditive---generalisation of an intensity matrix. Much as in the original case, this lower transition rate operator gives rise to a corresponding lower transition operator $\lt_t$, which is a non-linear---superadditive---generalisation of a stochastic matrix. For every real-valued function $f$ on $\states$, $\lt_tf$ is completely determined by the non-linear differential equation
\begin{equation}\label{eq:differentialinintroduction}
\frac{d}{dt}\lt_tf=\lrate \lt_tf \text{~~for all $t\geq0$,}
\end{equation}
with boundary condition $\lt_0f=f$~\cite{Skulj:2015cq}. The aim of this paper is to study the properties of this operator $\lt_t$ and, in particular, its limit behaviour as $t$ approaches infinity. Our main contribution---see Theorem~\ref{theo:finalequivalence}---is a simple necessary and sufficient condition for $\lrate$ to be ergodic, in the sense that for all real-valued functions $f$ on $\states$, $\lim_{t\to+\infty}\lt_tf$ exists and is constant. 

Our motivation for studying this property, and the reason for this papers title, is that $\lt_tf(x)$ can be interpreted as the conditional lower expectation $\underline{\mathbb{E}}(f(X_t)\vert X_0=x)$ of an imprecise continuous-time Markov chain, which, basically, is a set of continuous-time Markov chains whose possibly time-dependent transition rate matrix $Q_t$ is partially specified, in the sense that all that we known about it is that it takes values in some given set of transition rate matrices $\mathcal{Q}$.\footnote{In fact, $\lt_tf$ can also be interpreted as the conditional lower expectation $\underline{\mathbb{E}}(f(X_t)\vert X_0=x)$ that corresponds to an even larger set of stochastic processes, which, loosely speaking, is a set of (not necessarily Markovian) stochastic processes whose (possibly time-and history-dependent) transition rate matrix is not exactly specified, but is only known to take values in $\rateset$; see Reference~\cite{Krak:2016} for more information.} Indeed, as recently shown in References~~\cite{Krak:2016,Skulj:2015cq}, for the largest such set of Markov chains, and under relatively mild conditions on $\rateset$,\footnote{It should have separately specified rows, which means that every row has a set of possible candidate rows, and that the set of rate matrices $\rateset$ is constructed by combining these candidate rows in all possible ways, by selecting one element from each candidate set.} the tightest possible lower bound on the conditional expectation $\mathbb{E}(f(X_t)\vert X_0=x)$---the conditional lower expectation $\underline{\mathbb{E}}(f(X_t)\vert X_0=x)$---is equal to the solution $\lt_tf$ of the differential Equation~\eqref{eq:differentialinintroduction}, with $\lrate$ the lower envelope of $\mathcal{Q}$.


Therefore, if $\lim_{t\to+\infty}\lt_tf$ exists and is constant---if $\lrate$ is ergodic---this can be interpreted to mean that the limit value of the conditional lower expectation $\underline{\mathbb{E}}(f(X_t)\vert X_0=x)$ does not depend on the initial state $x$, or equivalently, that the imprecise continuous-time Markov chain under study has a unique limiting lower expectation operator $\underline{\mathbb{E}}_{\infty}$, in the sense that
\begin{equation*}
\underline{\mathbb{E}}_{\infty}(f)=\lim_{t\to+\infty}\underline{\mathbb{E}}(f(X_t)\vert X_0=x)
\text{~~~for all $x\in\states$ and all real functions $f$ on $\states$.}\vspace{7pt}
\end{equation*}
This interpretation turns our main result---a necessary and sufficient condition for $\lrate$ to be ergodic---into a practical tool: it provides a simple criterion for checking whether or not a given imprecise continuous-time Markov chain has a unique limiting lower expectation operator $\underline{\mathbb{E}}_{\infty}$. In the special case where the lower transition rate operator $\lrate$ is actually a transition rate matrix $Q$, our notion of ergodicity coincides with the usual one and, in that case, our results can be used to check wether the continuous-time Markov chain that corresponds to $Q$ has a unique limiting distribution $P_\infty$, whose expectation operator $\mathbb{E}_{\infty}$ will then be equal to $\underline{\mathbb{E}}_{\infty}$.



That being said, this paper does not adopt any specific interpretation, but takes a purely mathematical point of view. Our object of study here is the solution $\lt_t$ of the differential Equation~\eqref{eq:differentialinintroduction}, and our main result is a necessary and sufficient condition for $\lrate$ to be ergodic, in the sense that $\lt_t$ converges to a limit operator that maps functions to constants. As explained above, this result is directly applicable to---and inspired by---the theory of imprecise continuous-time Markov chains; more information about this field of study can be found in References~\cite{Krak:2016,Skulj:2015cq,Troffaes:CTMCapp,2012arXiv1204.0467S}.
However, we think that our results should also be of interest to other fields whose aim it is to robustify the theory of continuous-time Markov chains, such as continuous-time Markov decision processes~\cite{Xianping:2009}, continuous-time controlled Markov chains~\cite{guo2003} and interval continuous-time Markov chains~\cite{Galdino:2013}. More generally, we believe that our ideas and results are relevant to any theory that studies---or requires---some kind of robust generalisation of the matrix exponential of an intensity matrix.

We end this introduction with a brief overview of the structure of this paper. After Section~\ref{sec:prelim}, in which we introduce some basic preliminary concepts, the rest of this paper is structured as follows.

We start in Section~\ref{sec:lowertrans} by introducing the concept of a lower transition operator $\lt$, which is a non-linear---superadditive---generalisation of a stochastic matrix; the operator $\lt_t$ that is studied in this paper is a special case. We provide a definition, explain the connection with coherent lower previsions~\cite{Troffaes:2014tl,Walley:1991vk}, and use this connection to establish a number of technical properties. 

Section~\ref{sec:lowertransergodic} then goes on to define ergodicity for lower transition operators, which is  a discrete-time version of the notion of ergodicity that we study in this paper, and recalls that a lower transition operator $\lt$ will exhibit this type of ergodicity if and only if it is regularly absorbing~\cite{Hermans:2012ie}.
We also introduce a new property, called being $1$-step absorbing, and show that it is a sufficient condition for $\lt$ to be ergodic.

Next, in Section~\ref{sec:lowertransrate}, we introduce the concept of a lower transition rate operator $\lrate$, which, as already mentioned before, is a non-linear---superadditive---generalisation of an intensity matrix. We provide a definition, prove a number of properties, and establish a connection with lower transition operators.

Having introduced all of these related concepts and their properties, the rest of this paper focusses on our main object of interest, which is the time-dependent lower transition operator $\lt_t$ that corresponds to a given lower transition rate operator $\lrate$. Section~\ref{sec:differential} defines this operator as the unique solution to Equation~\eqref{eq:differentialinintroduction}, shows that it is indeed a lower transition operator, and then proves that it also satisfies another---closely related---differential equation, which applies directly to $\lt_t$ rather than $\lt_tf$. We end this section by establishing a limit expression for $\lt_t$, which resembles---and generalises---the well-known limit expression of a matrix exponential.

With these characterisations of $\lt_t$ in hand, Section~\ref{sec:lowertransrateergodic} then moves on to study its limit behaviour, and in particular, its ergodicity. First of all, we show that $\lrate$ is ergodic---$\lim_{t\to+\infty}\lt_t f$ exists and is constant---if and only if, for any $t>0$, $\lt_t$ is ergodic in the discrete-time sense of Section~\ref{sec:lowertransergodic}. Secondly, for any $t>0$, we show that $\lt_t$ is regularly absorbing if and only if it is $1$-step absorbing. Thirdly, we establish a simple qualitative method for checking whether $\lt_t$ is $1$-step absorbing; this method does not depend on $t$, and is expressed directly in terms of the lower transition rate operator $\lrate$. Finally, we explain how these three results, when combined, lead to a simple necessary and sufficient condition for $\lrate$ to be ergodic. 
All that is needed in order to check this condition, is the sign of a limited number of evaluations of $\lrate$. 

Section~\ref{sec:conclusions} concludes this paper. It briefly discusses our main result and then goes on to suggest some ideas for future research, including a number of specific conjectures and open questions that we consider to be important. The proofs of all our results are gathered in Appendix~\ref{app:proofs}; they are organised per section and in order of appearence. The appendix also contains some additional technical lemmas.

\section{Preliminaries}\label{sec:prelim}

Consider some finite \emph{state space} $\states$. 
Let $\gamblesX$ be the set of all real-valued functions on $\states$. For any $S\in\states$, let $\ind{S}\in\gamblesX$ be the indicator of $S$, defined by $\ind{S}(x)\coloneqq1$ if $x\in S$ and $\ind{S}(x)\coloneqq0$ otherwise. If $S$ is a singleton $\{x\}$, we also write $\ind{x}$ instead of $\ind{\{x\}}$. We use $I$ to denote the indentity map that maps any $f\in\gamblesX$ to itself. $\nats$ is the set of natural numbers without zero and $\natswith\coloneqq\nats\cup\{0\}$.

For any $f\in\gamblesX$, we let $\norm{f}\coloneqq\norm{f}_{\infty}\coloneqq\max\{\abs {f(x)}\colon x\in\states\}$ be the maximum norm. For any operator $A$ from $\gamblesX$ to $\gamblesX$ that is non-negatively homogeneous, meaning that
\begin{equation*}
A(\lambda f)=\lambda A(f)\text{ for all $f\in\gamblesX$ and all $\lambda\geq0$,}
\end{equation*}
we consider the induced operator norm
\begin{equation}\label{eq:operatornorm}
\norm{A}\coloneqq\sup\{\norm{Af}\colon f\in\gamblesX,\norm{f}=1\}.
\end{equation}

\noindent
Not only do these norms satisfy the usual defining properties of a norm, they also satisfy the following additional properties; see Appendix~\ref{app:prelimproofs} for a proof. For all $f\in\gamblesX$ and all operators $A,B$ from $\gamblesX$ to $\gamblesX$ that are non-negatively homogeneous:
\vspace{2pt}

\begin{enumerate}[label=N\arabic*:,ref=N\arabic*]
\item
$\norm{Af}\leq\norm{A}\norm{f}$
\label{N:gambleproduct}
\item
$\norm{AB}\leq\norm{A}\norm{B}$\label{N:operatorproduct}
\end{enumerate}




\section{Lower transition operators}\label{sec:lowertrans}

The first type of non-negatively homogeneous operator that we will consider in this paper is a lower transition operator~$\lt$. As we will show in Section~\ref{sec:differential}, the solution $\lt_t$ of the differential equation that we study in this paper is of this type.

\begin{definition}[Lower transition operator] A lower transition operator $\lt$ is a map from $\gamblesX$ to $\gamblesX$ such that for all $f,g\in\gamblesX$ and $\lambda\geq0$:
\vspace{4pt}
\begin{enumerate}[label=\emph{L\arabic*:},ref=L\arabic*]
\item
$\lt f\geq\min f$;\label{L:bound}
\item
$\lt(f+g)\geq\lt(f)+\lt(g)$\label{L:subadditivity}; \hfill\emph{[superadditivity]}
\item
$\lt(\lambda f)=\lambda\lt(f)$\label{L:homo}. \hfill\emph{[non-negative homogeneity]}
\end{enumerate}
\vspace{1pt}
The corresponding upper transition operator $\ut$ is defined by
\begin{equation}
\ut f\coloneqq-\lt(-f)\text{~~for all $f\in\gamblesX$.}\label{eq:Tconjugacy}
\end{equation}
\end{definition}

Basically, a lower transition operator is just a superadditive generalisation of a stochastic matrix. If the superadditivity axiom is replaced by an additivity axiom, a lower transition operator will coincide with its upper transition operator, and can then be identified with a stochastic matrix $T$.


For every lower transition operator $\lt$ and any $x\in\states$, the operator $\lt(\cdot)(x)$ is a coherent lower prevision~\cite{Troffaes:2014tl,Walley:1991vk}: a superadditive, non-negatively homogeneous map from $\gamblesX$ to $\reals$ that dominates the $\min$-operator. Therefore, lower transition operators are basically just finite vectors of coherent lower previsions. As a direct consequence, the following properties are implied by the corresponding versions for coherent lower previsions; see Reference~\cite[2.6.1]{Walley:1991vk}. For any $f,g\in\gamblesX$ and $\mu\in\reals$ and all sequences $\{f_n\}_{n\in\nats}\subseteq\gamblesX$:
\vspace{1pt}
\begin{enumerate}[label=L\arabic*:,ref=L\arabic*]
\setcounter{enumi}{3}
\item
$\min f\leq\lt f\leq\ut f\leq\max f$;\label{L:bounds}
\item
$\lt(f+\mu)=\lt(f)+\mu$;\label{L:constantadditivity}
\item
$f\geq g~\then~\lt(f)\geq\lt(g)$ and $\ut(f)\geq\ut(g)$;\label{L:monotone}
\item
$\abs{\lt f-\lt g}\leq\ut(\abs{f-g})$;\label{L:absinequality}
\item
$f_n\to f~\then~\lt f_n\to \lt f$.
\label{L:continuous}
\end{enumerate}
\vspace{2pt}
As a rather straightforward consequence of \ref{L:bounds} and \ref{L:absinequality}, we also find that
\vspace{4pt}
\begin{enumerate}[label=L\arabic*:,ref=L\arabic*]
\setcounter{enumi}{8}
\item
$\norm{\lt}\leq1$;\label{L:lowerthanone}
\item
$\norm{\lt f-\lt g}\leq\norm{f-g}$;\label{L:gambleinequality}
\item
$\norm{\lt A-\lt B}\leq\norm{A-B}$,\label{L:operatorinequality}
\end{enumerate}
\vspace{2pt}
where $A$ and $B$ are non-negatively homogeneous operators from $\gamblesX$ to $\gamblesX$; see Appendix~\ref{app:lowertransproofs} for a proof. Finally, as this next result establishes, a sequence of lower transition operators convergences pointwise if and only if it converges with respect to the operator norm.

\begin{proposition}\label{prop:Luniform}
For any lower transition operator $\lt$ and any sequence $\{\lt_n\}_{n\in\nats}$ of lower transition operators:
\begin{equation*}
\lt_n\to\lt~\asa~\forall f\in\gamblesX\colon\lt_nf\to\lt f.
\end{equation*}
\end{proposition}

\section{Ergodicity for lower transition operators}\label{sec:lowertransergodic}

In the linear case, that is, if the lower transition operator $\lt$ is actually a stochastic matrix $T$, then under rather weak assumptions, $T^n$ converges to a limit matrix that has identical rows, or equivalently, for all $f\in\gamblesX$, $\lim_{n\to+\infty}T^nf$ exists and is a constant function. This property of $T$ is called ergodicity,\footnote{This terminology is not universally adopted; we follow Senata~\cite[p.128]{seneta06}. Some authors use ergodicity to refer to a stronger property, which additionally requires that the identical rows of $\lim_{n\to+\infty}T^n$ consist of strictly positive elements, and which can be shown to be equivalent to the existence of some $n\in\nats$ such that $T^n$ consists of strictly positive elements only.} and the conditions under which it happens are well-studied; see for example Reference~\cite[Section 4.2]{seneta06}.

For our present purposes, we are interested in a generalised version of this concept of ergodicity, which applies to lower transition operators.

\begin{definition}[Ergodic lower transition operator]\label{def:lowertransergodic}
A lower transition operator $\lt$ is ergodic if, for all $f\in\gamblesX$, $\lim_{n\to\infty}\lt^n f$ exists and is a constant function.
\end{definition}

\noindent
Similarly, the corresponding upper transition operator $\ut$ is said to be ergodic if, for all $f\in\gamblesX$, $\lim_{n\to\infty}\ut^n f$ exists and is a constant function. It follows from Equation~\eqref{eq:Tconjugacy} that both notions are equivalent: $\lt$ is ergodic if and only if $\ut$ is. 

Hermans and De Cooman characterised this notion of ergodicity in Reference~\cite{Hermans:2012ie}, showing that a lower transition operator is ergodic if and only if it is regularly absorbing; see Proposition~\ref{prop:filip} further on. The following definition of a regularly absorbing lower transition operator is an equivalent but slightly simplified version of theirs; Lemma~\ref{lemma:simplifyfilip} in Appendix~\ref{app:lowertransergodicproofs} establishes the equivalence.

\begin{definition}[Regularly absorbing lower transition operator]\label{def:regularlyabsorbing}
A lower transition operator $\lt$ is regularly absorbing if it satisfies the following two conditions:
\begin{equation*}
\mathcal{X}_{\mathrm{RA}}\coloneqq
\{
x\in\states\colon
(\exists n\in\nats)~
\min\ut^n\ind{x}>0
\}\neq\emptyset\vspace{-3pt}
\end{equation*}
and
\begin{equation*}
(\forall x\in\states\setminus\mathcal{X}_{\mathrm{RA}})
(\exists n\in\nats)~
\lt^n\ind{\mathcal{X}_{\mathrm{RA}}}(x)>0.\vspace{5pt}
\end{equation*}
The first condition is called top class regularity and the second condition is called top class absorption.
\end{definition}

\begin{proposition}\label{prop:filip}
A lower transition operator $\lt$ is ergodic if and only if it is regularly absorbing.
\end{proposition}

If a lower transition operator satisfies Definition~\ref{def:regularlyabsorbing} with $n\coloneqq1$, we call this lower transition operator 1-step absorbing.

\begin{definition}[1-step absorbing lower transition operator]\label{def:onestepabsorbing}
A lower transition operator $\lt$ is 1-step absorbing if it satisfies the following two conditions:
\begin{equation*}
\mathcal{X}_{\mathrm{1A}}\coloneqq
\{
x\in\states\colon
\min\ut\ind{x}>0
\}\neq\emptyset\vspace{-3pt}
\end{equation*}
and
\begin{equation*}
(\forall x\in\states\setminus\mathcal{X}_{\mathrm{1A}})~
\lt\ind{\mathcal{X}_{\mathrm{1A}}}(x)>0.\vspace{5pt}
\end{equation*}
\end{definition}

Since $\mathcal{X}_{\mathrm{1A}}$ is clearly subset of $\mathcal{X}_{\mathrm{RA}}$, it follows from~\ref{L:monotone} that $\lt\ind{\mathcal{X}_{\mathrm{RA}}}\geq\lt\ind{\mathcal{X}_{\mathrm{1A}}}$, and therefore, every 1-step absorbing lower transition operator is guaranteed to be regularly absorbing as well. By combining this observation with Proposition~\ref{prop:filip}, it follows that being $1$-step absorbing is a sufficient condition for ergodicity.
However, in general, this stronger condition of being $1$-step absorbing is not necessary for ergodicity. The reason why we are nevertheless interested in this stronger property is because, as we will show further on in Section~\ref{sec:lowertransrateergodic}, for the particular lower transition operators $\lt_t$ that are the focus of this paper, both of these properties---Definitions~\ref{def:regularlyabsorbing} and~\ref{def:onestepabsorbing}---are equivalent; see Proposition~\ref{prop:regulariffonestep}.

\section{Lower transition rate operators}\label{sec:lowertransrate}

Having introduced a non-linear generalisation of a stochastic matrix, we now move on to introduce a similar generalisation of an intensity matrix---a matrix that has non-negative off-diagonal elements and rows that sum to zero. Again, the only difference is the additivity axiom, which we relax by replacing it with a superadditivity axiom. 

\begin{definition}[Lower transition rate operator] A lower transition rate operator $\lrate$ is a map from $\gamblesX$ to $\gamblesX$ such that for all $f,g\in\gamblesX$, $\lambda\geq0$, $\mu\in\reals$ and $x,y\in\states$:
\vspace{5pt}
\begin{enumerate}[label=\emph{R\arabic*:},ref=R\arabic*]
\item
$\lrate(\mu)=0$;\label{LR:constant}
\item
$\lrate(f+g)\geq\lrate(f)+\lrate(g)$;\label{LR:subadditive}\hfill\emph{[superadditivity]}
\item
$\lrate(\lambda f)=\lambda\lrate(f)$;\label{LR:homo}\hfill\emph{[non-negative homogeneity]}
\item
$x\neq y~\then~\lrate(\ind{y})(x)\geq0$.\label{LR:nondiagpositive}
\end{enumerate}
\vspace{1pt}
The corresponding upper transition operator $\urate$ is defined by
\begin{equation}
\urate f\coloneqq-\lrate(-f)\text{~for all $f\in\gamblesX$.}\label{eq:Qconjugacy}
\end{equation}
\end{definition}

As a rather straightforward consequence of this definition, a lower transition rate operator also satisfies the following properties; see Appendix~\ref{app:lowertransrateproofs} for a proof. For all $f\in\gamblesX$, $\mu\in\reals$ and $x\in\states$:
\vspace{4pt}
\begin{enumerate}[label=R\arabic*:,ref=R\arabic*]
\setcounter{enumi}{4}
\item $\lrate(f)\leq\urate(f)$;\label{LR:lowerbelowupper}
\item $\lrate(f+\mu)=\lrate(f)$;\label{LR:removeconstant}
\item $\urate(\ind{x})(x)\leq0$;\label{LR:diagonalnegative}
\item $2\norm{f}\lrate(\ind{x})(x)
\leq(f(x)-\min f)\lrate(\ind{x})(x)\leq\lrate(f)(x)$;\label{LR:gamblebound}
\item $\norm{\lrate}\leq2\max_{x\in\states}\abs{\lrate(\ind{x})(x)}$.\label{LR:operatorbound}
\end{enumerate}
\vspace{2pt}

Lower transition rate operators are very closely related to lower transition operators: they can be derived from each other. The following two results make this explicit.

\begin{proposition}\label{prop:fromQtoL}
Let $\lrate$ be a lower transition rate operator. Then for all $\Delta\geq0$ such that $\Delta\norm{\lrate}\leq1$, $I+\Delta\lrate$ is a lower transition operator.
\end{proposition}

\begin{proposition}\label{prop:fromLtoQ}
Let $\lt$ be a lower transition operator. Then for all $\Delta>0$, $\lrate\coloneqq\nicefrac{1}{\Delta}(\lt-I)$ is a lower transition rate operator.
\end{proposition}

Because of this connection, we can use results for lower transition operators to obtain similar results for lower transition rate operators. The following properties can for example be derived from~\ref{L:continuous}, \ref{L:gambleinequality} and~\ref{L:operatorinequality} respectively; see Appendix~\ref{app:lowertransrateproofs} for a proof. For any sequence $\{f_n\}_{n\in\nats}\subseteq\gamblesX$ and all $f,g\in\gamblesX$:
\vspace{4pt}
\begin{enumerate}[label=R\arabic*:,ref=R\arabic*]
\setcounter{enumi}{9}
\item
$f_n\to f~\then~\lrate f_n\to \lrate f$;\label{LR:continuous}
\item
$\norm{\lrate f-\lrate g}\leq 2\norm{\lrate}\norm{f-g}$;\label{LR:gambleinequality}
\item
$\norm{\lrate A-\lrate B}\leq 2\norm{\lrate}\norm{A-B}$,\label{LR:operatorinequality}
\end{enumerate}
\vspace{4pt}
where $A$ and $B$ are non-negatively homogeneous operators from $\gamblesX$ to $\gamblesX$. Similarly, the following result can be derived from Proposition~\ref{prop:Luniform}.

\begin{proposition}\label{prop:Quniform}
For any lower transition rate operator $\lrate$ and any sequence $\{\lrate_n\}_{n\in\nats}$ of lower transition rate operators:
\begin{equation*}
\lrate_n\to\lrate~\asa~\forall f\in\gamblesX\colon\lrate_nf\to\lrate f.
\end{equation*}
\end{proposition}

\section{The differential equation of interest}\label{sec:differential}

With all of the above material in place, we are now ready to introduce our main object of study: the time-dependent operator $\lt_t$ that corresponds to a given lower transition rate operator.

Let $\lrate$ be an arbitrary lower transition rate operator. Then for any $t\geq0$, we let $\lt_t$ be a map from $\gamblesX$ to $\gamblesX$, defined for all $f\in\gamblesX$ by the differential equation
\begin{equation}\label{eq:differential}
\frac{d}{dt}\lt_t f=\lrate \lt_t f\text{~~for all $t\geq0$}
\end{equation}
and the boundary condition $\lt_0 f\coloneqq f$. This definition is justified by a recent result of {\v{S}}kulj~\cite{Skulj:2015cq}, who showed that the above differential equation has a unique solution for all $t\geq0$. 

If $\lrate$ is additive, or equivalently, if $\lrate$ can be identified with an intensity matrix $Q$, then $\lt_t$ is equal to its matrix exponential $\smash{e^{Qt}}$. In the general case, the operator $\lt_t$ can be regarded as a superadditive generalisation of the matrix exponential. The rest of this section presents a number of basic properties of this operator and establishes some alternative characterisations for it.

First of all, as a direct consequence of its definition, we find that $\lt_t$ satisfies the following semigroup property:
\begin{equation}\label{eq:decomposition}
\lt_{t_1+t_2}=\lt_{t_1}\lt_{t_2}\text{ for all $t_1,t_2\geq0$}.
\end{equation}
Secondly, as already suggested by our notation, $\lt_t$ is a lower transition operator.

\begin{proposition}\label{prop:islowertransoperator}
Let $\lrate$ be a lower transition rate operator. Then for all $t\geq0$, $\lt_t$ is a lower transition operator.
\end{proposition}

Thirdly, as our next result establishes, we do not need to consider the above differential equation for every $f\in\gamblesX$ separately. Instead, we can apply a similar differential equation to the operator $\lt_t$ itself.

\begin{proposition}\label{prop:operatordifferential}
Let $\lrate$ be a lower transition rate operator. 
Then $\lt_0=I$ 
 and
\begin{equation}\label{eq:operatordifferential}
\frac{d}{dt}\lt_t =\lrate \lt_t\text{~~for all $t\geq0$},\vspace{3pt} 
\end{equation}
where the derivative is taken with respect to the operator norm.
\end{proposition}

Finally, $\lt_t$ can also be defined directly, without any reference to a differential equation. The following simple limit expression resembles---and generalises---the well-known limit definition of a matrix exponential.

\begin{proposition}\label{prop:approximation}
Let $\lrate$ be a lower transition rate operator. Then
\begin{equation*}
\lt_t=\lim_{n\to\infty}(I+\frac{t}{n}\lrate)^n
\end{equation*}
for all $t\geq0$.
\end{proposition}

The operator $\lt_t$ also satisfies some additional properties, some of which are stated and proved in Appendices~\ref{app:differentialproofs} and~\ref{app:lowertransrateergodicproofs}. However, since these properties are rather technical, and because we only need them in our proofs, we have chosen not to include them in the main text. Nevertheless, some of these properties---especially those that are stated in Proposition~\ref{prop:minmaxproperty} and Corollary~\ref{corol:constantsignprob}---may be of independent interest to the reader.

\section{Ergodicity for lower transition rate operators}\label{sec:lowertransrateergodic}



 Having introduced our main object of study in the previous section, we now move on to study its limit behaviour and, in particular, the conditions under which $\lrate$ is ergodic. In the linear case, that is, if $\smash{\lrate}$ can be identified with an intensity matrix $Q$, then $Q$ is said to be ergodic if $\smash{e^{Qt}}$ converges to a matrix that has identical rows,\footnote{Again, as was the case for the discrete-time version that we discussed in Section~\ref{sec:lowertransergodic}, our use of this terminology is not universally adopted; our definition is equivalent to that of Tornamb\`e~\cite[Definition~4.17]{tornambè1995discrete}. There are also authors who use ergodicity to refer to a stronger property, which additionally requires that the identical rows of $\lim_{t\to+\infty}\smash{e^{Qt}}$ consist of strictly positive elements.} or equivalently, if for all $f\in\gamblesX$, $\lim_{t\to+\infty}\smash{e^{Qt}}f$ exists and is a constant function. We generalise this property to the non-linear case in the following way.

\begin{definition}[Ergodic lower transition rate operator]\label{def:lowertransrateergodic}
A lower transition rate operator $\lrate$ is ergodic if, for all $f\in\gamblesX$, $\lim_{t\to\infty}\lt_tf$ exists and is a constant function.
\end{definition}

As we explained in the introduction, this property is particularly important in the context of imprecise continuous-time Markov chains, as it can then be interpreted to mean that such an imprecise continuous-time Markov chain converges to a unique limiting distribution that does not depend on the initial state. However, for the purposes of this paper, it is not necessary to understand the details of this interpretation. Instead, we will regard ergodicity as a purely mathematical property. The main contribution of this section---and, more generally, this paper---is a simple necessary and sufficient condition for a lower transiton operator $\lrate$ to be ergodic.

Our first step towards finding this condition is to link the continuous-time type of ergodicity that is considered in Definition~\ref{def:lowertransrateergodic} to the discrete-time version that we discussed in Section~\ref{sec:lowertransergodic}. Our next result establishes that $\lrate$ is ergodic in the sense of Definition~\ref{def:lowertransrateergodic} if and only if, for some arbitrary but fixed time $t>0$, the operator $\lt_t$ is ergodic in the sense of Definition~\ref{def:lowertransergodic}.

\begin{proposition}\label{prop:firstequivalences}
Let $\lrate$ be a lower transition rate operator. Then for any $t>0$, $\lrate$ is ergodic if and only if $\lt_t$ is ergodic.
\end{proposition}

At first sight---at least to us---this result is rather surprising. Since the ergodicity of $\lrate$ is a property that depends on the evolution of $\lt_t$ as $t$ approaches infinity, one would not suspect such a property to be completely determined by the features of a single operator $\lt_t$, on an arbitrary time point $t>0$. Nevertheless, as the above result shows, this is indeed the case.

By combining this result with Proposition~\ref{prop:filip}, we immediately obtain the following alternative characterisation of ergodicity.

\begin{corollary}\label{corol:ergodiciffdiscreteregularlyabsorbing}
Let $\lrate$ be a lower transition rate operator. Then for any $t>0$, $\lrate$ is ergodic if and only if $\lt_t$ is regularly absorbing.
\end{corollary}

This result is clearly a good first step in obtaining a simple charaterisation of ergodicity. Indeed, due to this result, instead of having to compute---or approximate---the limit behaviour of $\lt_t$ as $t$ approaches infinity, it now suffices to restrict attention to a single time point $t>0$, which we can even choose ourselves, and to check whether for this time point $t$, the operator $\lt_t$ is regularly absorbing. Furthermore, as our next result establishes, checking whether this particular type of lower transition operator is regularly absorbing is easier than it is for general lower transition operators: in this special case, being regularly absorbing is equivalent to being $1$-step absorbing.

\begin{proposition}\label{prop:regulariffonestep}
Let $\lrate$ be a lower transition rate operator. Then for any $t\geq0$, $\lt_t$ is regularly absorbing if and only if it is 1-step absorbing.
\end{proposition}

By combining this result with Corollary~\ref{corol:ergodiciffdiscreteregularlyabsorbing}, we immediately obtain yet another necessary and sufficient condition for $\lrate$ to be ergodic.

\begin{corollary}\label{corol:ergodiciffdiscrete1stepabsorbing}
Let $\lrate$ be a lower transition rate operator. Then for any $t>0$, $\lrate$ is ergodic if and only if $\lt_t$ is $1$-step absorbing.
\end{corollary}

Because of this result, checking whether $\lrate$ is ergodic is now reduced to checking whether $\lt_t$ is $1$-step absorbing, for some arbitrary but fixed $t>0$. Although this is already easier than studying the limit behaviour of $\lt_t$ directly, it is still non-trivial. As can be seen from Definition~\ref{def:onestepabsorbing}, it requires us to evaluate the strict positivity of numbers that are of the form $\ut_{\hspace{-1.5pt}t}\ind{x}(y)$ and $\lt_t\ind{A}(x)$, with $x,y\in\states$ and $A\subseteq\states$. At first sight, this still seems to be a rather cumbersome task that will involve either solving the differential Equation~\eqref{eq:differential} or applying the limit expression in Proposition~\ref{prop:approximation}. However, as it turns out, this is not the case.

Indeed, as we are about to show, the strict positivity of $\ut_{\hspace{-1.5pt}t}\ind{x}(y)$ and $\lt_t\ind{A}(x)$ does not depend on the specific value of $t$, but only on the lower transition operator $\lrate$. 
In order to make this specific, we introduce the following notions of upper and lower reachability.

\begin{definition}[Upper reachability]\label{def:upperreachability}
For any $x,y\in\states$, we say that $x$ is upper reachable from $y$, and denote this by \,$\smash{y\dotabovearrow x}$,
if there is some sequence $y=x_0,\dots,x_n=x$ such that, for all $k\in\{1,\dots,n\}$:
\begin{equation*}
x_k\neq x_{k-1}\text{ and }\urate(\ind{x_k})(x_{k-1})>0.
\vspace{2mm}
\end{equation*}
\end{definition}

\begin{definition}[Lower reachability]\label{def:lowerreachability}
For any $x\in\states$ and $A\subseteq\states$, we say that $A$ is lower reachable from $x$, and denote this by $x\dotbelowarrow A$,
 if $x\in A_n$, where $\{A_k\}_{k\in\natswith}$ is an increasing sequence that is defined by $A_0\coloneqq A$ and
\begin{equation}\label{eq:lowerreachability}
A_{k+1}\coloneqq A_k\cup\{y\in\states\setminus A_k\colon \lrate(\ind{A_k})(y)>0\}
\text{~~for all $k\in\natswith$},
\end{equation}
and where $n$ is the first index such that $A_n=A_{n+1}$.
\end{definition}

An important property of both of these two notions is that they are easy to check. For upper reachability, it suffices to draw a directed graph that has the elements of $\states$ as its nodes and which features an arrow from $y$ to $x$ if and only if $\urate(\ind{x})(y)>0$. Checking whether $y$ is upper reachable from $x$ is then clearly equivalent to checking whether it is possible to start in $x$ and follow the arrows in the graph to reach $y$. This is a standard reachability problem that can either be solved manually, or by means of techniques from graph theory. Lower reachability essentially requires us to construct the sequence $\{A_k\}_{k\in\natswith}$ up to the index $n$. Since it follows from the increasing nature of this sequence that $n\leq\abs{\states\setminus A}$, this too is a straightforard task.


The reason why we are interested in these notions of lower and upper reachability are the following two equivalences.

\begin{proposition}\label{prop:iffuperreachable}
Let $\lrate$ be a lower transition rate operator. Then for any $t>0$ and any $x,y\in\states$: 
\begin{equation*}
\ut_{\hspace{-1.5pt}t}\ind{x}(y)>0
~~\Leftrightarrow~~
y\dotabovearrow x.
\vspace{2mm}
\end{equation*}
\end{proposition}

\begin{proposition}\label{prop:ifflowerreachable}
Let $\lrate$ be a lower transition rate operator. Then for any $t>0$, any $x\in\states$ and any $A\subseteq\states$:
\begin{equation*}
\lt_t\ind{A}(x)>0
~~\Leftrightarrow~~
x\dotbelowarrow A.
\vspace{2mm}
\end{equation*}
\end{proposition}

By combining these equivalences with Definition~\ref{def:onestepabsorbing} and Corollary~\ref{corol:ergodiciffdiscrete1stepabsorbing}, we easily obtain the following result, which is the  characterisation of ergodicty that we have been after all along.

\begin{theorem}\label{theo:finalequivalence}
A lower transition rate operator $\lrate$ is ergodic if and only if
\vspace{1pt}
\begin{equation*}
\mathcal{X}_{\mathrm{1A}}\coloneqq
\{
x\in\states\colon
(\forall y\in\states)~y\dotabovearrow x
\}\neq\emptyset\vspace{-7pt}
\end{equation*}
and
\vspace{0pt}
\begin{equation*}
(\forall x\in\states\setminus\mathcal{X}_{\mathrm{1A}})~~x\dotbelowarrow \mathcal{X}_{\mathrm{1A}}.\vspace{6pt}
\end{equation*}
\end{theorem}

We consider this neccesary and sufficient condition for the ergodicity of $\lrate$ to be the main contribution of this paper. The reason why it is to be preferred over other necessary and sufficient conditions, such as those that are given in Corollaries~\ref{corol:ergodiciffdiscreteregularlyabsorbing} and~\ref{corol:ergodiciffdiscrete1stepabsorbing}, is because it does not require us to evaluate the operator $\lt_t$. Instead, all we have to do is solve a limited number of lower and upper reachability problems, which, as can be seen from Definitions~\ref{def:upperreachability} and~\ref{def:lowerreachability}, only requires us to evalutate the operator $\lrate$. This is obviously preferable, because $\lrate$ is directly available, whereas $\lt_t$ is known only indirectly through the differential Equation~\eqref{eq:differential} or the limit expression in Proposition~\ref{prop:approximation}. 


\section{Conclusions}\label{sec:conclusions}

The main contribution of this paper is a simple necessary and sufficient condition for the ergodicity of a lower transition rate operator $\lrate$. Specifically, as can be seen from Theorem~\ref{theo:finalequivalence}, it is necessary and sufficient for at least one state $x$ to be upper reachable from every other state $y$, and for the set $\mathcal{X}_{\mathrm{1A}}$ of all the states that satisfy this condition to be lower reachable from each of the states that does not.
The main conclusion that can be drawn from this result is that  ergodicity is easily satisfied. For example, it already suffices---but is by no means necessary---for every state to be upper reachable from any other state.

This result provides us with a simple method for checking wether $\lrate$ is ergodic, or equivalently, whether $\lt_tf$ is guaranteed to converge to a constant function as $t$ approaches infinity. The reason why this is important to us is because, as explained in the introduction, in the context of imprecise continuous-time Markov chains, this notion of ergodicity is equivalent to the existence of a unique limiting distribution that is independent of the initial state. Therefore, our results can be used to check whether or not such a unique limiting distribution exists.

Although the existence of such a limiting distribution is important, it is of course only one of the many aspects of the limit behaviour of imprecise continuous-time Markov chains. Many problems still remain unsolved. For example: what happens if we no longer care about the influence of the initial state? In particular: for a given initial state, under which conditions will an imprecise continuous-time Markov chain converge to a limiting distribution that is allowed to depend on this initial state? Or equivalently, using the terminolgy of this paper: which conditions does $\lrate$ need to satisfy in order for $\lim_{t\to+\infty}\lt_tf(x)$ to exist? Ergodicity is clearly a sufficient condition---since it additionaly requires that this limit does not depend of $x$---but it is definitely not necessary. In fact, we even conjecture that this type of convergence requires no conditions at all.

 The simple reason why we suspect this conjecture to hold is because, as mentioned in the introduction, if $\lrate$ is an intensity matrix $Q$, then rather remarkably, regardless of the specific intensity matrix $Q$ that is considered, $\smash{\lt_t=e^{Qt}}$ will always converge to a limit~\cite[Theorem~II.10.1]{Chung:1967}. By analogy, for any lower transition rate operator $\lrate$, we think that the corresponding lower transition operator $\lt_t$ will always converge to a limit. Investigating wether this conjecture is indeed true would be a nice topic for future research.

Another interesting line of future research would be to study ergodicity---or other types of convergence---from a quantitave rather than just qualitative point of view, by also taking into account the rate of convergence. For the discrete-time type of ergodicity that we discussed in Section~\ref{sec:lowertransergodic}, such a study has already been conducted in References~\cite{Hermans:2012ie,Skulj:2011db}, leading to the development of a coefficient of ergodicity that simultaneously captures both the qualitative aspect of convergence---``does it converge or not?''---and the quantitative aspect---``at which rate does it converge?''. We think that similar coefficients of ergodicity can also be developed for the continuous-time models that we have considered in this paper.

Finally, we would like to point out that these suggestions for future research are just the tip of the iceberg, because they focus solely on the limit behaviour of imprecise continuous-time Markov chains. Ultimately, we hope that our contributions will serve as a first step towards a further theoretic development of the general field of imprecise continuous-time Markov chains. 
The reason why we consider such developments to be important is because, given the succes of precise continuous-time Markov chains in various fields of application~\cite{Anderson:1991}, and the ever increasing demand for features such as reliability and robustness in these applications, we are convinced that imprecise continuous-time Markov chains have plenty of applied potential. Nevertheless, almost no applications have been developed so far. It seems to us that one of the main reasons for this lack of applications is a severe lack of available theoretical tools. We hope that a further theoretical development of the field of imprecise continuous-time Markov chains will allow this field to flourish, and will turn it into a full-fledged robust extension of the field of continous-time Markov chains.

\section*{Acknowledgements}
Jasper De Bock is a Postdoctoral Fellow of the Research Foundation - Flanders (FWO) and wishes to acknowledge its financial support. 
The author would also like to thank Gert de Cooman, Matthias C. M. Troffaes, Stavros Lopatatzidis and Thomas Krak for stimulating discussions on the topic of imprecise continuous-time Markov chains.

\bibliography{general}

\begin{thebibliography}{10}
\expandafter\ifx\csname url\endcsname\relax
  \def\url#1{\texttt{#1}}\fi
\expandafter\ifx\csname urlprefix\endcsname\relax\def\urlprefix{URL }\fi
\expandafter\ifx\csname href\endcsname\relax
  \def\href#1#2{#2} \def\path#1{#1}\fi

\bibitem{Chung:1967}
K.~L. Chung, Markov chains with stationary transition probabilities, Die
  Grundlehren der mathematischen Wissenschaften in Einzeldarstellungen,
  Springer, Berlin, New York, 1967.

\bibitem{Anderson:1991}
W.~J. Anderson, Continuous-Time Markov chains, an Applications-Oriented
  Approach, Springer Series in Statistics, Springer, New York, 1991.

\bibitem{Skulj:2015cq}
D.~{\v{S}}kulj, {Efficient computation of the bounds of continuous time
  imprecise Markov chains}, Applied mathematics and computation 250~(C) (2015)
  165--180.

\bibitem{Krak:2016}
T.~Krak, J.~De~Bock, {Imprecise Continuous-Time Markov Chains}Work in progress.

\bibitem{Troffaes:CTMCapp}
M.~C.~M. Troffaes, J.~Gledhill, D.~{\v S}kulj, S.~Blake, Using imprecise
  continuous time markov chains for assessing the reliability of power networks
  with common cause failure and non-immediate repair, in: ISIPTA '15:
  Proceedings of the Ninth International Symposium on Imprecise Probability:
  Theories and Applications, 2015, pp. 287--294.

\bibitem{2012arXiv1204.0467S}
D.~{{\v S}kulj}, {Interval matrix differential equations} (2012).
\newblock \href {http://arxiv.org/abs/1204.0467} {\path{arXiv:1204.0467}}.

\bibitem{Xianping:2009}
X.~Guo, O.~Hern\'andez-Lerma, {Continuous-time Markov decision processes:
  theory and applications.}, Stochastic Modelling and Applied Probability 62,
  Springer, Berlin, 2009.

\bibitem{guo2003}
X.~Guo, O.~Hernández-Lerma, Continuous-time controlled markov chains, Annals
  of Applied Probability 13~(1) (2003) 363--388.

\bibitem{Galdino:2013}
S.~Galdino, Interval continuous-time markov chains simulation, in: Proceedings
  of the 1013 International Conference on Fuzzy Theory and Its Applications,
  2013, pp. 273--278.

\bibitem{Troffaes:2014tl}
M.~C.~M. Troffaes, G.~de~Cooman, {Lower previsions}, John Wiley {\&} Sons,
  2014.

\bibitem{Walley:1991vk}
P.~Walley, {Statistical reasoning with imprecise probabilities}, Chapman and
  Hall, London, 1991.

\bibitem{Hermans:2012ie}
F.~Hermans, G.~de~Cooman, {Characterisation of ergodic upper transition
  operators}, International Journal of Approximate Reasoning 53~(4) (2012)
  573--583.

\bibitem{seneta06}
E.~Seneta, Non-negative matrices and Markov chains, Springer, New York, 2006.

\bibitem{tornambè1995discrete}
A.~Tornamb{\`e}, Discrete-event System Theory: An Introduction, World
  Scientific, 1995.

\bibitem{Skulj:2011db}
D.~{\v{S}}kulj, R.~Hable, {Coefficients of ergodicity for Markov chains with
  uncertain parameters}, Metrika 76~(1) (2011) 107--133.

\bibitem{Royden:2010vn}
H.~L. Royden, P.~M. Fitzpatrick, {Real Analysis}, 4th Edition, Prentice Hall,
  2010.

\end{thebibliography}

\appendix

\section{Proofs}\label{app:proofs}

\subsection{Proofs of results in Section~\ref{sec:prelim}}\label{app:prelimproofs}

Let $A$ and $B$ be two non-negatively homogeneous operators from $\gamblesX$ to $\gamblesX$ and consider any $f,g\in\gamblesX$ and $\lambda\in\reals$.

It is well known that the maximum norm on $\gamblesX$ satisfies the defining properties of a norm: it is absolutely homogeneous ($\norm{\lambda f}=\abs{\lambda}\norm{f}$), it is subadditive ($\norm{f+g}\leq\norm{f}+\norm{g}$) and it separates points ($\norm{f}=0\then f=0$). The induced operator norm also satisfies these properties. Firstly, it is absolutely homogeneous because the maximum norm is: 
\begin{align*}
\norm{\lambda A}&=\sup\{\norm{\lambda Af}\colon f\in\gamblesX,\norm{f}=1\}\\
&=\sup\{\abs{\lambda}\norm{Af}\colon f\in\gamblesX,\norm{f}=1\}\\
&=\abs{\lambda}\sup\{\norm{Af}\colon f\in\gamblesX,\norm{f}=1\}
=\abs{\lambda}\norm{A}.
\end{align*}
Secondly, it is subadditive because the maximum norm is: 
\begin{align*}
\norm{A+B}&=\sup\{\norm{(A+B)f}\colon f\in\gamblesX,\norm{f}=1\}\\
&=\sup\{\norm{Af+Bf}\colon f\in\gamblesX,\norm{f}=1\}\\
&\leq\sup\{\norm{Af}+\norm{Bf}\colon f\in\gamblesX,\norm{f}=1\}\\
&\leq\sup\{\norm{A}+\norm{B}\colon f\in\gamblesX,\norm{f}=1\}
=\norm{A}+\norm{B}.
\end{align*}
Thirdly, it separates points because the maximum norm does: if $\norm{A}=0$, then $A=0$ because, for all $f\in\gamblesX$, it follows from \ref{N:gambleproduct}---which we will prove next---that
\begin{equation*}
0\leq\norm{Af}\leq\norm{A}\norm{f}=0
\end{equation*}
and therefore, since the maximum norm separates points, that $Af=0$.

In order to prove \ref{N:gambleproduct}, we consider two cases: $f=0$ and $f\neq0$. If $f\neq0$, or equivalently, if $\norm{f}\neq0$, we let $g\coloneqq\nicefrac{f}{\norm{f}}$. If $f=0$, or equivalently, if $\norm{f}=0$, we let $g\coloneqq 1$. In both cases, this guarantees that $f=\norm{f}g$ and $\norm{g}=1$ and therefore, we find that
\begin{equation*}
\norm{Af}=\norm{A(\norm{f}g)}
=\norm{\norm{f}Ag}
=\norm{f}\norm{Ag}\leq\norm{f}\norm{A},
\end{equation*}
where the inequality holds because $\norm{f}\geq0$ and $\norm{Ag}\leq\norm{A}$.

Finally, \ref{N:operatorproduct}  follows rather easily from \ref{N:gambleproduct}:
\begin{align*}
\norm{AB}&=\sup\{\norm{ABf}\colon f\in\gamblesX,\norm{f}=1\}\\
&\leq\sup\{\norm{A}\norm{Bf}\colon f\in\gamblesX,\norm{f}=1\}\\
&=\norm{A}\sup\{\norm{Bf}\colon f\in\gamblesX,\norm{f}=1\}=\norm{A}\norm{B}.
\end{align*}

\subsection{Proofs of results in Section~\ref{sec:lowertrans}}\label{app:lowertransproofs}

\begin{proof}[Proof of \ref{L:lowerthanone}, \ref{L:gambleinequality} and \ref{L:operatorinequality}]
\ref{L:lowerthanone} follows from Equation~\eqref{eq:operatornorm} because we know from \ref{L:bounds} that $\norm{\lt f}\leq\norm{f}$ for all $f\in\gamblesX$.
\ref{L:gambleinequality} follows from \ref{L:absinequality} and \ref{L:bounds} (in that order). \ref{L:operatorinequality} follows from Equation~\eqref{eq:operatornorm} and \ref{L:gambleinequality}.
\end{proof}

\begin{proof}[Proof of Proposition~\ref{prop:Luniform}]
The direct implication follows trivially from \ref{N:gambleproduct}. For the converse implication, we provide a proof by contradiction. 
Assume that  $\lt_nf\to\lt f$ for all $f\in\gamblesX$. Assume \emph{ex absurdo} that $\lt_n\not\to\lt$. Then since $\lt_n\not\to\lt$, it follows that $\limsup_{n\to\infty}\norm{\lt_n-\lt}>0$, which implies that there is some $\epsilon>0$ and an increasing sequence $n_k$, $k\in\nats$, of natural numbers such that $\norm{\lt_{n_k}-\lt}>\epsilon$ for all $k\in\nats$. Furthermore, for all $k\in\nats$, it follows from $\norm{\lt_{n_k}-\lt}>\epsilon$ and Equation~\eqref{eq:operatornorm} that there is some $f_k\in\gamblesX$ such that $\norm{f_k}=1$ and $\norm{\lt_{n_k}f_k-\lt f_k}>\epsilon$. 
Since the sequence $f_k$, $k\in\nats$, is clearly bounded---because $\norm{f_k}=1$---it follows from the Bolzano-Weierstrass theorem that it has a convergent subsequence, which implies that there is some $f\in\gamblesX$ and an increasing sequence $k_i$, $i\in\nats$, of natural numbers such that $\lim_{i\to\infty}\norm{f_{k_i}-f}=0$. Furthermore, since we have assumed that $\lt_nf\to\lt f$, it follows that
\begin{equation*}
\lim_{i\to\infty}\norm{\lt_{n_{k_i}}f-\lt f}=\lim_{n\to\infty}\norm{\lt_{n}f-\lt f}=0.
\end{equation*}
Hence, since it follows from \ref{L:gambleinequality} that
\begin{align*}
\norm{\lt_{n_{k_i}}f_{k_i}-\lt f_{k_i}}
&=
\norm{\big(\lt_{n_{k_i}}f_{k_i}-\lt_{n_{k_i}}f\big)
+
\big(\lt_{n_{k_i}}f-\lt f\big)
+
\big(\lt f-\lt f_{k_i}\big)
}\\
&\leq
\norm{\lt_{n_{k_i}}f_{k_i}-\lt_{n_{k_i}}f}
+
\norm{\lt_{n_{k_i}}f-\lt f}
+
\norm{\lt f-\lt f_{k_i}}
\\
&\leq
\norm{f_{k_i}-f}
+
\norm{\lt_{n_{k_i}}f-\lt f}
+
\norm{f_{k_i}-f},
\end{align*}
we find that
\begin{equation*}
\lim_{i\to\infty}\norm{\lt_{n_{k_i}}f_{k_i}-\lt f_{k_i}}=0.
\end{equation*}
Since $\norm{\lt_{n_{k_i}}f_{k_i}-\lt f_{k_i}}>\epsilon>0$ for all $i\in\nats$, this is a contradiction.
\end{proof}

\subsection{Proofs of results in Section~\ref{sec:lowertransergodic}}\label{app:lowertransergodicproofs}

\begin{lemma}\label{lemma:simplifyfilip}
A lower transition operator $\lt$ is regularly absorbing if and only if
\begin{equation*}
\mathcal{X}'_{\mathrm{RA}}\coloneqq
\{
x\in\states\colon
(\exists n\in\nats)(\forall k\geq n)~
\min\ut^k\ind{x}>0
\}\neq\emptyset
\end{equation*}
and
\begin{equation*}
(\forall x\in\states\setminus\mathcal{X}'_{\mathrm{RA}})
(\exists n\in\nats)~
\ut^n\ind{\mathcal{X}\setminus\mathcal{X}'_{\mathrm{RA}}}(x)<1.
\end{equation*}
Furthermore, the set $\mathcal{X}'_{\mathrm{RA}}$ is equal to the set $\mathcal{X}_{\mathrm{RA}}$ that was used in Definition~\ref{def:regularlyabsorbing}.
\end{lemma}

\begin{proof}[Proof of Lemma~\ref{lemma:simplifyfilip}]
Consider any $x\in\mathcal{X}_{\mathrm{RA}}$. Definition~\ref{def:regularlyabsorbing} then implies that there is some $n\in\nats$ such that $\min\ut^n\ind{x}>0$, and therefore, because of \ref{L:bounds}, we know that $\ut^{n+1}\ind{x}=\ut(\ut^n\ind{x})\geq\min\ut^n\ind{x}>0$, which implies that $\min\ut^{n+1}\ind{x}>0$. In the same way, we also find that $\min\ut^{n+2}\ind{x}>0$ and, by continuing in this way, that $\min\ut^k\ind{x}>0$ for all $k\geq n$. Since this holds for all $x\in\mathcal{X}_{\mathrm{RA}}$, it follows that $\mathcal{X}_{\mathrm{RA}}\subseteq\mathcal{X}'_{\mathrm{RA}}$. Since $\mathcal{X}'_{\mathrm{RA}}$ is clearly a subset of $\mathcal{X}_{\mathrm{RA}}$, this implies that $\mathcal{X}'_{\mathrm{RA}}=\mathcal{X}_{\mathrm{RA}}$. Hence, trivially, $\mathcal{X}_{\mathrm{RA}}\neq\emptyset$ if and only if $\mathcal{X}'_{\mathrm{RA}}\neq\emptyset$. The result now follows because it holds for all $x\in\states\setminus\mathcal{X}'_{\mathrm{RA}}=\states\setminus\mathcal{X}_{\mathrm{RA}}$ and all $n\in\nats$ that
\begin{equation*}
\ut^n\ind{\states\setminus\mathcal{X}'_{\mathrm{RA}}}(x)
=\ut^n(1-\ind{\mathcal{X}'_{\mathrm{RA}}})(x)
=1-\lt^n(\ind{\mathcal{X}'_{\mathrm{RA}}})(x)
=1-\lt^n(\ind{\mathcal{X}_{\mathrm{RA}}})(x),
\end{equation*}
where the second equality follows from \ref{L:constantadditivity} and Equation~\eqref{eq:Tconjugacy}.
\end{proof}

\begin{proof}[Proof of Proposition~\ref{prop:filip}]
Since we know from Lemma~\ref{lemma:simplifyfilip} that our definition of a regularly absorbing lower transition operator is equivalent to the definition in Reference~\cite{Hermans:2012ie}, this result is identical to \cite[Proposition~3]{Hermans:2012ie}.
\end{proof}

\subsection{Proofs of results in Section~\ref{sec:lowertransrate}}\label{app:lowertransrateproofs}

\begin{proof}[Proof of \ref{LR:lowerbelowupper}, \ref{LR:removeconstant}, \ref{LR:diagonalnegative}, \ref{LR:gamblebound} and \ref{LR:operatorbound}]
\ref{LR:lowerbelowupper} holds because it follows from Equation~\eqref{eq:Qconjugacy}, \ref{LR:subadditive} and \ref{LR:constant} that
\begin{equation*}
\lrate(f)-\urate(f)=\lrate(f)+\lrate(-f)\leq\lrate(f-f)=\lrate(0)=0.
\end{equation*}
\ref{LR:removeconstant} holds because it follows from \ref{LR:constant} and \ref{LR:subadditive} that
\begin{equation*}
\lrate(f)
=\lrate(f)+\lrate(\mu)
\leq\lrate(f+\mu)
=\lrate(f+\mu)+\lrate(-\mu)
\leq\lrate(f).
\end{equation*}
\ref{LR:diagonalnegative} holds because it follows from Equation~\eqref{eq:Qconjugacy}, \ref{LR:removeconstant}, \ref{LR:subadditive} and \ref{LR:nondiagpositive}---in that order---that
\begin{align*}
\urate(\ind{x})(x)
=-\lrate(-\ind{x})(x)
&=-\lrate(1-\ind{x})(x)\\
&=-\lrate(\textstyle\sum_{y\in\states\setminus\{x\}}\ind{y})(x)
\leq-\textstyle\sum_{y\in\states\setminus\{x\}}\lrate(\ind{y})(x)
\leq0.
\end{align*}
\ref{LR:gamblebound} holds because it follows from \ref{LR:removeconstant}, \ref{LR:subadditive}, \ref{LR:homo}, \ref{LR:nondiagpositive}, \ref{LR:diagonalnegative} and \ref{LR:lowerbelowupper}---in that order---that
\begin{align*}
\lrate(f)(x)
=\lrate(f-\min f)(x)
&\geq\textstyle\sum_{y\in\states}\lrate\big((f(y)-\min f)\ind{y}\big)(x)\\
&=\textstyle\sum_{y\in\states}(f(y)-\min f)\lrate(\ind{y})(x)\\
&\geq(f(x)-\min f)\lrate(\ind{x})(x)\\
&\geq(\max f-\min f)\lrate(\ind{x})(x)
\geq2\norm{f}\lrate(\ind{x})(x).
\end{align*}
We end by proving \ref{LR:operatorbound}. Consider any $g\in\gamblesX$ such that $\norm{g}=1$. It then follows from \ref{LR:gamblebound} and \ref{LR:diagonalnegative} that
\begin{equation*}
\lrate(g)\geq2\norm{g}\min_{x\in\states}\lrate(\ind{x})(x)
\geq-2\max_{x\in\states}\abs{\lrate(\ind{x})(x)}.
\end{equation*}
Similarly, since $\norm{-g}=\norm{g}=1$, we also find that $\lrate(-g)\geq-2\max_{x\in\states}\abs{\lrate(\ind{x})(x)}.$ By combining these two inequalities with \ref{LR:lowerbelowupper} and Equation~\eqref{eq:Qconjugacy}, it follows that
\begin{equation*}
-2\max_{x\in\states}\abs{\lrate(\ind{x})(x)}\leq\lrate(g)\leq\urate(g)=-\lrate(-g)\leq2\max_{x\in\states}\abs{\lrate(\ind{x})(x)},
\end{equation*}
which implies that $\norm{\lrate(g)}\leq2\max_{x\in\states}\abs{\lrate(\ind{x})(x)}$. Since this is true for all $g\in\gamblesX$ such that $\norm{g}=1$, \ref{LR:operatorbound} now follows from Equation~\eqref{eq:operatornorm}.
\end{proof}

\begin{proof}[Proof of Proposition~\ref{prop:fromQtoL}]
\ref{L:subadditivity} and~\ref{L:homo} follow trivially from \ref{LR:subadditive} and~\ref{LR:homo}. We only prove \ref{L:bound}. Consider any $f\in\gamblesX$. Then
\begin{align*}
(I+\Delta\lrate)f
&=f+\Delta\lrate f\\
&=f+\Delta\sum_{x\in\states}\ind{x}\lrate(f)(x)\\
&\geq f+\Delta\sum_{x\in\states}(f(x)-\min f)\ind{x}\lrate(\ind{x})(x)\\
&\geq f-\Delta\sum_{x\in\states}(f(x)-\min f)\ind{x}\norm{\lrate(\ind{x})}\\
&\geq f-\Delta\sum_{x\in\states}(f(x)-\min f)\ind{x}\norm{\lrate}\\
&=f-\Delta\norm{\lrate}(f-\min f)
=(f-\min f)(1-\Delta\norm{\lrate})+\min f\geq\min f,
\end{align*}
where the first inequality follows from \ref{LR:gamblebound} and the third inequality follows from Equation~\eqref{eq:operatornorm} and the fact that $\norm{\ind{x}}=1$.
\end{proof}

\begin{proof}[Proof of Proposition~\ref{prop:fromLtoQ}]
Simply check each of the defining properties:
\ref{LR:constant} holds because \ref{L:bounds} implies that $\lt(\mu)=\mu$ for all $\mu\in\reals$, \ref{LR:subadditive} follows from \ref{L:subadditivity}, \ref{LR:homo} follows from \ref{L:homo} and \ref{LR:nondiagpositive} follows from \ref{L:bound}. 
\end{proof}

\begin{proof}[Proof of \ref{LR:continuous}, \ref{LR:gambleinequality} and \ref{LR:operatorinequality}]
\ref{LR:continuous}, \ref{LR:gambleinequality} and \ref{LR:operatorinequality} are trivial if $\lrate=0$. Therefore, we may assume that $\lrate\neq0$, which implies that $\norm{\lrate}>0$. Now let $\lt\coloneqq I+\nicefrac{1}{\norm{\lrate}}\lrate$. It then follows from Proposition~\ref{prop:fromQtoL} that $\lt$ is a lower transition operator.
We first prove~\ref{LR:continuous}. If $f_n\to f$, then $\lt f_n\to\lt f$ because of~\ref{L:continuous}. Since $\lrate=\norm{\lrate}(\lt-I)$, this implies that $\lrate f_n\to\lrate$. \ref{LR:gambleinequality} holds because 
\begin{align*}
\norm{\lrate f-\lrate g}
&=\norm{\norm{\lrate}(\lt f-f)-\norm{\lrate}(\lt g-g)}\\
&\leq\norm{\lrate}\norm{\lt f-\lt g}+\norm{\lrate}\norm{f-g}
\leq2\norm{\lrate}\norm{f-g},
\end{align*}
where the last inequality follows from~\ref{L:gambleinequality}. Similarly,
\ref{LR:operatorinequality} holds because
\begin{align*}
\norm{\lrate A-\lrate B}
&=\norm{\norm{\lrate}(\lt A-A)-\norm{\lrate}(\lt B-B)}\\
&\leq\norm{\lrate}\norm{\lt A-\lt B}+\norm{\lrate}\norm{A-B}
\leq2\norm{\lrate}\norm{A-B},
\end{align*}
where the last inequality follows from~\ref{L:operatorinequality}.
\end{proof}

\begin{proof}[Proof of Proposition~\ref{prop:Quniform}]
The direct implication follows trivially from~\ref{N:gambleproduct}. We only prove the converse implication. Assume that $\lrate_nf\to\lrate f$ for all $f\in\gamblesX$. 
For all $x\in\states$, this implies that $\lrate_n(\ind{x})(x)\to\lrate(\ind{x})(x)$, which in turn implies that there is some $c_x>0$ such that $\vert\lrate(\ind{x})(x)\vert<c_x$ and $\vert\lrate_n(\ind{x})(x)\vert<c_x$ for all $n\in\nats$. Let $c\coloneqq\max_{x\in\states}c_x$. It then follows from~\ref{LR:operatorbound} that $\vert\vert\lrate\vert\vert\leq2c$ and $\vert\vert\lrate_n\vert\vert\leq2c$ for all $n\in\nats$.
Choose any $0<\Delta\leq\nicefrac{1}{2c}$. It then follows from Proposition~\ref{prop:fromQtoL} that $\lt\coloneqq I+\Delta\lrate$ and $\lt_n\coloneqq I+\Delta\lrate_n$, $n\in\nats$, are lower transition operators. Furthermore, since $\lrate_nf\to\lrate f$ for all $f\in\gamblesX$, it follows that $\lt_nf\to\lt f$ for all $f\in\gamblesX$. By applying Proposition~\ref{prop:Luniform}, we now find that $\lt_n\to\lt$, which implies that $\lrate_n\to\lrate$ because
\begin{equation*}
\norm{\lrate_n-\lrate}
=\frac{1}{\Delta}\norm{\Delta\lrate_n-\Delta\lrate}
=\frac{1}{\Delta}\norm{(I+\Delta\lrate_n)-(I+\Delta\lrate)}
=\frac{1}{\Delta}\norm{\lt_n-\lt}.
\end{equation*}
\end{proof}

\subsection{Proofs of results in Section~\ref{sec:differential}}\label{app:differentialproofs}

\begin{lemma}\label{lemma:continouslydifferentiable}
Let $\lrate$ be a lower transition rate operator. Then for all $f\in\gamblesX$, $\lt_sf$ is continously differentiable on $[0,\infty)$.
\end{lemma}
\begin{proof}
It follows from Equation~\eqref{eq:differential} that $\lt_sf$ is continuous on $[0,\infty)$. Therefore, since $\lrate$ is a continuous operator [\ref{LR:continuous}], $\lrate\lt_sf$ is also continuous on $[0,\infty)$. Because of Equation~\eqref{eq:differential}, this implies that $\lt_sf$ is continuously differentiable on $[0,\infty)$.
\end{proof}

\begin{lemma}\label{lemma:minimumincreases}
Let $\lrate$ be a lower transition rate operator and let $\Gamma(s)$ be a continously differentiable map from $[0,t]$ to $\gamblesX$ for which $\smash{\frac{d}{ds}}\Gamma(s)\geq\lrate\Gamma(s)$ for all $s\in[0,t]$. Then $\min\Gamma(t)\geq\min\Gamma(0)$.
\end{lemma}
\begin{proof}
Since $\Gamma(s)$ is continously differentiable on $[0,t]$, it follows that for every $x\in\states$, $\Gamma(s)(x)$ is also continuously differentiable on $[0,t]$, which implies that it is absolutely continuous on $[0,t]$. Hence, since a minimum of a finite number of absolutely continuous functions is again absolutely continuous, we find that $\min\Gamma(s)$ is absolutely continuous on $[0,t]$, which implies---see Reference~\cite[Theorem 10, Section 6.5]{Royden:2010vn}---that $\min\Gamma(s)$ has a derivative $\frac{d}{ds}\min\Gamma(s)$  almost everwhere on $(0,t)$, that this derivative is Lebesgue integrable over $[0,t]$, and that
\begin{equation}\label{eq:lebesgue}
\min\Gamma(t)=\min\Gamma(0)+\int_0^t\Big(\frac{d}{ds}\min\Gamma(s)\Big)ds.
\end{equation}
Consider now any $t^*\in(0,t)$ for which $\min\Gamma(s)$ has a derivative and consider any $x\in\states$ for which $\Gamma(t^*)(x)=\min\Gamma(t^*)$ [clearly, there is at least one such $x$]. Since $\Gamma(s)(x)$ is differentiable, $\smash{\frac{d}{ds}}\Gamma(s)(x)$ exists in $t^*$. Assume \emph{ex absurdo} that $\smash{\frac{d}{ds}\Gamma(s)(x)\big\vert_{s=t^*}}$ is not equal to $\frac{d}{ds}\min\Gamma(s)\big\vert_{s=t^*}$ or, equivalently, that $\frac{d}{ds}(\Gamma(s)(x)-\min\Gamma(s))\big\vert_{s=t^*}\neq0$. Then, because $\Gamma(s)(x)-\min\Gamma(s)$ is continuous [since $\Gamma(s)(x)$ and $\min\Gamma(s)$ are both (absolutely) continuous] and because $t^*\in(0,t)$ and $\Gamma(t^*)(x)-\min\Gamma(t^*)=0$, it follows that there is some $t'\in(0,t)$ such that $\Gamma(t')(x)-\min\Gamma(t')<0$ or, equivalently, such that $\Gamma(t')(x)<\min\Gamma(t')$. Since this is clearly a contradiction, it follows that
\begin{equation}\label{eq:equalderivative}
\frac{d}{ds}\Gamma(s)(x)\Big\vert_{s=t^*}=\frac{d}{ds}\min\Gamma(s)\Big\vert_{s=t^*}.
\end{equation}
We also have that
\begin{equation*}
\frac{d}{ds}\Gamma(s)(x)\Big\vert_{s=t^*}
\geq\lrate\big(\Gamma(t^*)\big)(x)
\geq\big(\Gamma(t^*)(x)-\min\Gamma(t^*)\big)\lrate(\ind{x})(x)=0,
\end{equation*}
where the second inequality follows from \ref{LR:gamblebound} and the last equality follows because $\Gamma(t^*)(x)=\min\Gamma(t^*)$. By combining this result with Equation~\eqref{eq:equalderivative}, we find that, for all $t^*\in(0,t)$ for which $\min\Gamma(s)$ has a derivative, $\frac{d}{ds}\min\Gamma(s)\big\vert_{s=t^*}\geq0$. It therefore follows from Equation~\eqref{eq:lebesgue} that $\min\Gamma(t)\geq\min\Gamma(0)$.
\end{proof}

\begin{proof}[Proof of Proposition~\ref{prop:islowertransoperator}]
We first prove \ref{L:bound}. Consider any $f\in\gamblesX$. It then follows from Lemma~\ref{lemma:continouslydifferentiable} that $\lt_sf$ is continuously differentiable on $[0,t]$. Therefore, and because of Equation~\eqref{eq:differential}, we infer from Lemma~\ref{lemma:minimumincreases} that $\min\lt_tf\geq\min\lt_0f$. Since $\lt_0f=f$, this implies that $\min\lt_tf\geq\min f$, which in turn implies that $\lt_tf\geq\min f$.

Let us now prove~\ref{L:subadditivity}. Consider any $f,g\in\gamblesX$. 
It follows from Lemma~\ref{lemma:continouslydifferentiable} that $\lt_sf$, $\lt_sg$ and $\lt_s(f+g)$ are continuously differentiable on $[0,t]$, which implies that $\Gamma(s)\coloneqq\lt_s(f+g)-\lt_sf-\lt_sg$ is continuously differentiable on $[0,t]$.
Furthermore, for all $s\in[0,t]$, it follows from Equation~\eqref{eq:differential} and~\ref{LR:subadditive} that
\begin{align*}
\frac{d}{ds}\Gamma(s)
&=\frac{d}{ds}\lt_s(f+g)-\frac{d}{ds}\lt_sf-\frac{d}{ds}\lt_sg\\
&=\lrate\lt_s(f+g)-\lrate\lt_sf-\lrate\lt_sg\\
&=\lrate\big(\Gamma(s)+\lt_sf+\lt_sg\big)-\lrate\lt_sf-\lrate\lt_sg
\geq\lrate\Gamma(s).
\end{align*}
 Therefore, we infer from Lemma~\ref{lemma:minimumincreases} that $\min\Gamma(t)\geq\min\Gamma(0)$. Since $\Gamma(0)=\lt_0(f+g)-\lt_0f-\lt_0g=0$, this implies that $\min\Gamma(t)\geq0$, which in turn implies that $\Gamma(t)\geq0$ or, equivalently, that $\lt_t(f+g)\geq\lt_tf+\lt_tg$.

We end by proving~\ref{L:homo}. Consider any $f\in\gamblesX$ and $\lambda\geq0$. It then follows from Equation~\eqref{eq:differential} and \ref{LR:homo} that
\begin{equation*}
\frac{d}{ds}(\lambda\lt_sf)
=\lambda\frac{d}{ds}\lt_sf
=\lambda\lrate\lt_sf
=\lrate(\lambda\lt_sf)
\text{~for all $s\geq0$}.
\end{equation*}
Since we also have that $\lambda\lt_0f=\lambda f=\lt_0(\lambda f)$, it follows that $\lambda\lt_sf$ satisfies the same differential equation and boundary condition as $\lt_s(\lambda f)$. Since we know that this differential equation and boundary condition lead to a unique solution on $[0,\infty)$, it follows that $\lt_t(\lambda f)=\lambda\lt_t(f)$.
\end{proof}

\begin{lemma}\label{lemma:boundaryderivative}
Let $\lrate$ be a lower transition rate operator. Then 
\begin{equation*}
\lim_{\Delta\to0^+}\lt_\Delta=I
\text{ and }
\lim_{\Delta\to0^+}\nicefrac{1}{\Delta}(\lt_\Delta-I)=\lrate.
\end{equation*}
\end{lemma}
\begin{proof}
For any $f\in\gamblesX$, it follows from Equation~\eqref{eq:differential} that $\lt_tf$ is continuous on $[0,\infty)$, which implies that $\lim_{\Delta\to0^+}\lt_\Delta f=\lt_0f=f$. Therefore, we infer from Proposition~\ref{prop:Luniform} that $\lim_{\Delta\to0^+}\lt_\Delta=I$, which proves the first part of this lemma. We end by proving the second part. For any $f\in\gamblesX$, it follows from Equation~\eqref{eq:differential} that
\begin{equation*}
\lim_{\Delta\to0^+}\nicefrac{1}{\Delta}(\lt_\Delta -I)(f)
=\lim_{\Delta\to0^+}\nicefrac{1}{\Delta}(\lt_\Delta f -f)
=\lim_{\Delta\to0^+}\nicefrac{1}{\Delta}(\lt_\Delta f-\lt_0f)=\lrate\lt_0 f=\lrate f.
\end{equation*}
Therefore, and since, for all $\Delta>0$, $\nicefrac{1}{\Delta}(\lt_\Delta-I)$ is a lower transition rate operator because of Proposition~\ref{prop:fromLtoQ}, it follows from Proposition~\ref{prop:Quniform} that $\lim_{\Delta\to0^+}\nicefrac{1}{\Delta}(\lt_\Delta-I)=\lrate$.
\end{proof}

\begin{proof}[Proof of Proposition~\ref{prop:operatordifferential}]
Since $\lt_0f\coloneqq f$ for all $f\in\gamblesX$, it follows trivially that $\lt_0=I$.
Consider now any $t\geq0$. In order to prove that $\frac{d}{dt}\lt_t=\lrate\lt_t$, it suffices to show that for all $\epsilon>0$, there is some $\delta>0$ such that 
\begin{equation}\label{eq:operatorderivativehulp}
\norm{\frac{\lt_s-\lt_t}{s-t}-\lrate\lt_t}<\epsilon
\text{ for all $s\geq0$ such that $0<\abs{t-s}<\delta$.}
\end{equation}
So consider any $\epsilon>0$. 
If $\lrate=0$, Equation~\eqref{eq:operatorderivativehulp} is trivially true because, since $I$ clearly satisfies Equation~\eqref{eq:differential}, it follows from the unicity of the solution of Equation~\eqref{eq:differential} that $\lt_t=\lt_s=I$. Therefore, in the remainder of this proof, we may assume that $\lrate\neq0$, which implies that $\norm{\lrate}\neq0$. It then follows from Lemma~\ref{lemma:boundaryderivative} that there are $\delta_1>0$ and $\delta_2>0$ such that $\norm{\lt_q-I}<\nicefrac{\epsilon}{4\norm{\lrate}}$ for all $0<q<\delta_1$ and $\norm{\nicefrac{1}{\Delta}(\lt_\Delta-I)-\lrate}<\nicefrac{\epsilon}{2}$ for all $0<\Delta<\delta_2$. Now define $\delta\coloneqq\min\{\delta_1,\delta_2\}$ and consider any $s\geq0$ such that $0<\abs{t-s}<\delta$. Let $u\coloneqq\min\{s,t\}$, $\Delta\coloneqq\abs{t-s}$ and $q\coloneqq t-u$, which implies that $0\leq q\leq\Delta<\delta\leq\delta_1$ and $0<\Delta<\delta\leq\delta_2$. If $q=0$, then $\lt_q=\lt_0=I$ and therefore $\norm{\lrate\lt_q-\lrate}=\norm{\lrate-\lrate}=0$. If $q>0$, it follows from~\ref{LR:operatorinequality} and Proposition~\ref{prop:islowertransoperator} that $\norm{\lrate\lt_q-\lrate}\leq2\norm{\lrate}\norm{\lt_q-I}<2\norm{\lrate}\nicefrac{\epsilon}{4\norm{\lrate}}=\nicefrac{\epsilon}{2}$. Hence, in all cases, we find that $\norm{\lrate\lt_q-\lrate}<\nicefrac{\epsilon}{2}$. The result now holds because
\begin{align*}
\norm{\frac{\lt_s-\lt_t}{s-t}-\lrate\lt_t}
&=\norm{\frac{\lt_{\Delta+u}-\lt_u}{\Delta}-\lrate\lt_{q+u}}
=\norm{\frac{\lt_\Delta\lt_u-\lt_u}{\Delta}-\lrate\lt_q\lt_u}\\
&\leq\norm{\frac{\lt_{\Delta}-I}{\Delta}-\lrate\lt_{q}}\norm{\lt_u}
\leq\norm{\frac{\lt_{\Delta}-I}{\Delta}-\lrate\lt_{q}}\\
&\leq\norm{\frac{\lt_{\Delta}-I}{\Delta}-\lrate}+\norm{\lrate\lt_q-\lrate}
<\frac{\epsilon}{2}+\frac{\epsilon}{2}=\epsilon,
\end{align*}
where the second equality follows from Equation~\eqref{eq:decomposition}, the first inequality follows from \ref{N:operatorproduct} and the second inequality follows from Proposition~\ref{prop:islowertransoperator} and \ref{L:lowerthanone}.
\end{proof}

\begin{proof}[Proof of Proposition~\ref{prop:approximation}]
The result is trivial if $t=0$. In the remainder of this proof, we assume that $t>0$.
The result for $\lrate=0$ is also trivial because, since $I$ then clearly satisfies Equation~\eqref{eq:differential}, it follows from the unicity of the solution of Equation~\eqref{eq:differential} that $\lt_t=I$. Therefore, in the remainder of this proof, we assume that $\lrate\neq0$, which implies that $\norm{\lrate}\neq0$. We will now prove that for every $\epsilon>0$, there is some $n\in\nats$ such that
\begin{equation*}
\norm{\lt_t-(I+\frac{t}{k}\lrate)^k}<\epsilon \text{ for all $k\geq n$.}
\end{equation*}
So consider any $\epsilon>0$. It then follows from Lemma~\ref{lemma:boundaryderivative} that there is some $\delta>0$ such that $\norm{\nicefrac{1}{\Delta}(\lt_\Delta-I)-\lrate}<\nicefrac{\epsilon}{t}$ for all $0<\Delta<\delta$. 
Now choose $n\in\nats$ such that $n>\max\{\nicefrac{t}{\delta},t\norm{\lrate}\}$ and consider any $k\geq n$. Let $\Delta\coloneqq\nicefrac{t}{k}\leq\nicefrac{t}{n}$, which implies that $0<\Delta<\delta$ and $\Delta\norm{\lrate}<1$. Then
\begin{align*}
&\norm{(\lt_\Delta)^k-(I+\Delta\lrate)^k}\\
&\quad\quad=\norm{(\lt_\Delta)^k-(\lt_\Delta)^{k-1}(I+\Delta\lrate)+(\lt_\Delta)^{k-1}(I+\Delta\lrate)-(I+\Delta\lrate)^k}\\
&\quad\quad\leq\norm{(\lt_\Delta)^k-(\lt_\Delta)^{k-1}(I+\Delta\lrate)}+\norm{(\lt_\Delta)^{k-1}(I+\Delta\lrate)-(I+\Delta\lrate)^k}\\
&\quad\quad\leq\norm{\lt_\Delta-(I+\Delta\lrate)}+\norm{(\lt_\Delta)^{k-1}-(I+\Delta\lrate)^{k-1}}\norm{I+\Delta\lrate}\\
&\quad\quad\leq\norm{\lt_\Delta-(I+\Delta\lrate)}+\norm{(\lt_\Delta)^{k-1}-(I+\Delta\lrate)^{k-1}},
\end{align*}
where the second inequality follows from Proposition~\ref{prop:islowertransoperator} and \ref{L:operatorinequality} [by applying them repeatedly] and \ref{N:operatorproduct}, and the third inequality follows from \ref{L:lowerthanone} and Proposition~\ref{prop:fromQtoL}. By continuing in this way, we find that
\begin{equation*}
\norm{(\lt_\Delta)^k-(I+\Delta\lrate)^k}
\leq
k\norm{\lt_\Delta-(I+\Delta\lrate)}.
\end{equation*}
Therefore, since
\begin{equation*}
\norm{\lt_\Delta-(I+\Delta\lrate)}
=\Delta\norm{\frac{\lt_\Delta-I}{\Delta}-\lrate}
<\Delta\frac{\epsilon}{t}=\frac{t}{k}\frac{\epsilon}{t}=\frac{\epsilon}{k},
\end{equation*}
and because it follows from Equation~\eqref{eq:decomposition} that $\lt_t=(\lt_{\nicefrac{t}{k}})^k=(\lt_{\Delta})^k$, we find that
\begin{equation*}
\norm{\lt_t-(I+\frac{t}{k}\lrate)^k}
=\norm{(\lt_{\Delta})^k-(I+\Delta\lrate)^k}
\leq
k\norm{\lt_\Delta-(I+\Delta\lrate)}<k\frac{\epsilon}{k}=\epsilon.
\end{equation*}
\end{proof}

\subsection{Proofs of results in Section~\ref{sec:lowertransrateergodic}}\label{app:lowertransrateergodicproofs}

\begin{proof}[Proof of Proposition~\ref{prop:firstequivalences}]
First assume that $\lrate$ is ergodic. For all $f\in\gamblesX$, it then follows from Definition~\ref{def:lowertransrateergodic} that $\lim_{s\to\infty}\lt_s f$ exists and is a constant function. Therefore, for all $f\in\gamblesX$, it follows from Equation~\eqref{eq:decomposition} that
\begin{equation*}
\lim_{n\to\infty}\lt_t^n f=\lim_{n\to\infty}\lt_{nt} f=\lim_{s\to\infty}\lt_s f
\end{equation*}
exists and is a constant function, which implies that $\lt_t$ is ergodic. 

Next, assume that $\lt_t$ is ergodic. 
This means that, for all $f\in\gamblesX$, there is some $c_f\in\reals$ such that
\begin{equation}\label{eq:limitlts}
(\forall\epsilon>0)(\exists n\in\nats)(\forall k\geq n)
\norm{\lt_t^k f-c_f}<\epsilon.
\end{equation} 
Consider now any $f\in\gamblesX$ and any $\epsilon>0$. It then follows from Equation~\eqref{eq:limitlts} that there is some $n_\epsilon\in\nats$ such that $\norm{\lt_t^{n_\epsilon} f-c_f}<\epsilon$, which, because of Proposition~\ref{prop:islowertransoperator} and~\ref{L:constantadditivity}, implies that $\norm{\lt_t^{n_\epsilon}(f-c_f)}<\epsilon$. Now let $s_\epsilon\coloneqq n_\epsilon t$. Then for all $s\geq s_\epsilon$, we have that
\begin{equation*}
\norm{\lt_s f - c_f}
=\norm{\lt_s(f - c_f)}
=\norm{\lt_{s-s_\epsilon}\lt_{t}^{n_\epsilon}(f - c_f)}
\leq\norm{\lt_{s-s_\epsilon}}\norm{\lt_{t}^{n_\epsilon}(f - c_f)}
<\epsilon,
\end{equation*}
where the first equality follows from Proposition~\ref{prop:islowertransoperator} and \ref{L:constantadditivity}, the second equality follows from Equation~\eqref{eq:decomposition}, the first inequality follows from \ref{N:gambleproduct} and the last inequality follows from Proposition~\ref{prop:islowertransoperator}, \ref{L:lowerthanone} and the fact that $\norm{\lt_t^{n_\epsilon} f-c_f}<\epsilon$.
Hence, we have found that for all $\epsilon>0$, there is some $s_\epsilon>0$ such that $\norm{\lt_sf-c_f}<\epsilon$ for all $s\geq s_\epsilon$. In other words: $\lim_{s\to\infty}\lt_s f=c_f$. Since this is true for all \mbox{$f\in\gamblesX$}, it follows from Definition~\ref{def:lowertransrateergodic} that $\lrate$ is ergodic.
\end{proof}

\begin{proof}[Proof of Corollary~\ref{corol:ergodiciffdiscreteregularlyabsorbing}]
Immediate consequence of Propositions~\ref{prop:filip},~\ref{prop:islowertransoperator} and~\ref{prop:firstequivalences}.
\end{proof}

\begin{lemma}\label{lemma:stayingabovetheminimum}
Let $\lrate$ be a lower transition rate operator. Consider any $f\in\gamblesX$ and $x\in\states$ such that $f(x)>\min f$. Then for all $t\geq0$: $\lt_t f(x)>\min f$.
\end{lemma}
\begin{proof}
Since we know from Lemma~\ref{lemma:continouslydifferentiable} that $\lt_tf$ is continuously differentiable on $[0,\infty)$, we know that $r_t\coloneqq \lt_t f-\min f$ and therefore also 
$r_t(x)$ is continuously differentiable on $[0,\infty)$. 
Furthermore, for all $t\geq0$, it follows from Proposition~\ref{prop:islowertransoperator} and~\ref{L:bound} that $r_t\geq0$, which in turn implies that
\begin{equation*}
\frac{d}{dt}r_t(x)
=\frac{d}{dt}\lt_tf(x)
=\lrate(\lt_tf)(x)
=\lrate(r_t)(x)
\geq\sum_{y\in\states}r_t(y)\lrate(\ind{y})(x)
\geq r_t(x)\lrate(\ind{x})(x),
\end{equation*}
where the second equality follows from Equation~\eqref{eq:differential}, the third equality follows from \ref{LR:removeconstant}, the first inequality follows from~\ref{LR:subadditive} and~\ref{LR:homo} and the last inequality follows from \ref{LR:nondiagpositive}.
Hence, for all $t\geq0$, we find that $r_t(x)\geq r_0(x)e^{[\lrate(\ind{x})(x)]t}$.
Since we also know that $r_0(x)=\lt_0f(x)-\min f=f(x)-\min f>0$, this implies that for all $t\geq0$:
\begin{equation*}
\lt_tf(x)-\min f=r_t(x)\geq r_0(x)e^{[\lrate(\ind{x})(x)]t}>0.
\end{equation*}
\end{proof}

\begin{lemma}\label{lemma:increasing}
Let $\lrate$ be a lower transition rate operator. Consider any $f\in\gamblesX$, $x\in\states$ and $s\geq0$ such that $\lt_sf(x)>\min f$. Then for all $t\geq s$: $\lt_tf(x)>\min f$.
\end{lemma}
\begin{proof}
Because of Equation~\eqref{eq:decomposition}, it suffices to prove that $\lt_{t-s}\lt_s f(x)>\min f$. We consider two cases: $\min \lt_s f>\min f$ and $\min \lt_s f=\min f$; $\min \lt_s f<\min f$ is not possible because of Proposition~\ref{prop:islowertransoperator} and~\ref{L:bound}. If $\min \lt_s f>\min f$, it follows from Proposition~\ref{prop:islowertransoperator} and~\ref{L:bound} that $\lt_{t-s}\lt_s f(x)\geq\min \lt_s f>\min f$.
If $\min \lt_s f=\min f$, then $\lt_sf(x)>\min\lt_s f$ and therefore, because of Lemma~\ref{lemma:stayingabovetheminimum}, $\lt_{t-s}\lt_s f(x)>\min\lt_s f=\min f$.
\end{proof}


\begin{lemma}\label{lemma:constantsignlower}
Let $\lrate$ be a lower transition rate operator. Consider any $f\in\gamblesX$, $x\in\states$ and $t,s>0$. Then
\begin{equation*}
\lt_tf(x)>\min f\asa\lt_sf(x)>\min f.
\end{equation*}
\end{lemma}
\begin{proof}
For any $\tau\geq0$, let
\begin{equation}\label{eq:def:statestau}
\states_\tau\coloneqq\{y\in\states\colon \lt_\tau f(y)>\min f\}.
\end{equation}
It then follows from Lemma~\ref{lemma:increasing} that $\states_\tau$ is an increasing function of $\tau$:
\begin{equation}\label{eq:increasingfunction}
\tau\leq\tau'~\then~\states_{\tau}\subseteq\states_{\tau'}.
\end{equation}

\noindent
Assume \emph{ex absurdo} that
\begin{equation}\label{eq:neginitialconstant}
(\forall\tau'>0)~(\forall\states'\subseteq\states)~(\exists\tau\in(0,\tau'])~\states_\tau\neq\states'.
\end{equation}
Choose any $\tau_1>0$. Then clearly, $\states_{\tau_1}\subseteq\states$. Therefore, due to Equation~\eqref{eq:neginitialconstant}, we know that there is some $0<\tau_2<\tau_1$ such that $\states_{\tau_2}\neq\states_{\tau_1}$, which, because of Equation~\eqref{eq:increasingfunction}, implies that $\states_{\tau_2}\subset\states_{\tau_1}$. Similarly, we infer that there is some \mbox{$0<\tau_3<\tau_2$} such that \mbox{$\states_{\tau_3}\subset\states_{\tau_2}$}. By continuing in this way, we obtain an infinite sequence of time points $\tau_1>\tau_2>\tau_3>\cdots>\tau_i>\cdots>0$ such that $\states\supseteq\states_{\tau_1}\supset\states_{\tau_2}\supset\states_{\tau_3}\supset\cdots\supset\states_{\tau_i}\supset\cdots$. Since $\states$ is a finite set, this is a contradiction, leading us to conclude that Equation~\eqref{eq:neginitialconstant} is false. This implies that there is some $\tau^*>0$ and $\states^*\subseteq\states$ such that
\begin{equation}\label{eq:initialconstant}
(\forall\tau\in(0,\tau^*])~\states_\tau=\states^*.
\end{equation}

Fix any $\tau>\tau^*$ and choose $n\in\nats$ high enough such that $2\nicefrac{\tau}{n}\leq\tau^*$.
It then follows from Equation~\eqref{eq:initialconstant} that $\states_{\nicefrac{\tau}{n}}=\states_{2\nicefrac{\tau}{n}}=\states^*$. Furthermore, because of Proposition~\ref{prop:islowertransoperator},~\ref{L:bound} and Equation~\eqref{eq:def:statestau}, we know that $\lt_{\nicefrac{\tau}{n}}f(y)=\lt_{2\nicefrac{\tau}{n}}f(y)=\min f$ for all $y\in\states\setminus\states^*$.
Therefore, we infer from Equation~\eqref{eq:def:statestau} that there is some $\lambda>0$ such that
\begin{equation*}
\lt_{2\nicefrac{\tau}{n}}f-\min f\leq\lambda(\lt_{\nicefrac{\tau}{n}}f-\min f),
\end{equation*}
which, because of Equation~\eqref{eq:decomposition}, Proposition~\ref{prop:islowertransoperator} and \ref{L:constantadditivity} implies that
\begin{equation*}
\lt_{\nicefrac{\tau}{n}}^2(f-\min f)
=\lt_{2\nicefrac{\tau}{n}}(f-\min f)
=\lt_{2\nicefrac{\tau}{n}}f-\min f
\leq\lambda(\lt_{\nicefrac{\tau}{n}}f-\min f).
\end{equation*}
Hence, it follows from Proposition~\ref{prop:islowertransoperator}, \ref{L:constantadditivity}, Equation~\eqref{eq:decomposition},~\ref{L:monotone} and~\ref{L:homo} that
\begin{equation}\label{eq:constantinequality}
\lt_\tau f-\min f
=\lt_\tau (f-\min f)
=\lt_{\nicefrac{\tau}{n}}^n(f-\min f)
\leq\lambda^{n-1}(\lt_{\nicefrac{\tau}{n}}f-\min f).
\end{equation}
Consider now any $y\in\states\setminus\states^*$. Since $\lt_{\nicefrac{\tau}{n}}f(y)=\min f$, it follows from Equation~\eqref{eq:constantinequality} that $\lt_{\tau}f(y)\leq\min f$, which in turn implies that $y\notin\states_\tau$. Since this holds for all $y\in\states\setminus\states^*$, we find that $\states_\tau\subseteq\states^*=\states_{\nicefrac{\tau}{n}}$. Furthermore, since $\nicefrac{\tau}{n}\leq\tau$, it follows from Equation~\eqref{eq:increasingfunction} that $\states_{\nicefrac{\tau}{n}}\subseteq\states_\tau$. Hence, we find that $\states_\tau=\states^*$. Since this is true for all $\tau>\tau^*$, it follows from Equation~\eqref{eq:initialconstant} that
\begin{equation*}
\states_\tau=\states^*
\text{ for all $\tau>0$}.
\end{equation*}
Therefore, due to Equation~\eqref{eq:def:statestau}, we find that
\begin{equation*}
\lt_tf(x)>\min f
\asa x\in\states_t
\asa x\in\states_s
\asa\lt_sf(x)>\min f.
\end{equation*}
\end{proof}


\begin{lemma}\label{lemma:increasingupper}
Let $\lrate$ be a lower transition rate operator. Consider any $f\in\gamblesX$, $x\in\states$ and $s\geq0$ such that $\ut_sf(x)>\min f$. Then for all $t\geq s$: $\ut_tf(x)>\min f$.
\end{lemma}
\begin{proof}
Because of Equations~\eqref{eq:Tconjugacy} and~\eqref{eq:decomposition}, it suffices to prove that $\ut_{t-s}\ut_s f(x)>\min f$. 
We consider two cases: $\min \ut_s f>\min f$ and $\min \ut_s f=\min f$; $\min \ut_s f<\min f$ is not possible because of Proposition~\ref{prop:islowertransoperator} and~\ref{L:bounds}. 
If $\min \ut_s f>\min f$, it follows from Proposition~\ref{prop:islowertransoperator} and~\ref{L:bounds} that $\ut_{t-s}\ut_s f(x)\geq\min \ut_s f>\min f$.
If $\min \ut_s f=\min f$, then $\ut_sf(x)>\min\ut_s f$ and therefore, it follows from Proposition~\ref{prop:islowertransoperator}, \ref{L:bounds} and Lemma~\ref{lemma:stayingabovetheminimum} that $\ut_{t-s}\ut_s f(x)\geq\lt_{t-s}\ut_s f(x)>\min\ut_s f=\min f$.
\end{proof}

\begin{lemma}\label{lemma:constantsignupper}
Let $\lrate$ be a lower transition rate operator. Consider any $f\in\gamblesX$, $x\in\states$ and $t,s>0$. Then
\begin{equation*}
\ut_tf(x)>\min f\asa\ut_sf(x)>\min f.
\end{equation*}
\end{lemma}
\begin{proof}
For any $\tau\geq0$, let
\begin{equation}\label{eq:def:statestau:upper}
\states_\tau\coloneqq\{y\in\states\colon \ut_\tau f(y)>\min f\}.
\end{equation}
It then follows from Lemma~\ref{lemma:increasingupper} that $\states_\tau$ is an increasing function of $\tau$:
\begin{equation}\label{eq:increasingfunctionupper}
\tau\leq\tau'~\then~\states_{\tau}\subseteq\states_{\tau'}.
\end{equation}
Using an argument that is identical to that in Lemma~\ref{lemma:constantsignlower}, we find that this implies that there is some $\tau^*>0$ and $\states^*\subseteq\states$ such that
\begin{equation}\label{eq:initialconstant:upper}
(\forall\tau\in(0,\tau^*])~\states_\tau=\states^*.
\end{equation}

Fix any $\tau>\tau^*$ and choose $n\in\nats$ high enough such that $2\nicefrac{\tau}{n}\leq\tau^*$.
It then follows from Equation~\eqref{eq:initialconstant:upper} that $\states_{\nicefrac{\tau}{n}}=\states_{2\nicefrac{\tau}{n}}=\states^*$. 
Furthermore, because of Proposition~\ref{prop:islowertransoperator},~\ref{L:bounds} and Equation~\eqref{eq:def:statestau:upper}, we know that $\ut_{\nicefrac{\tau}{n}}f(y)=\ut_{2\nicefrac{\tau}{n}}f(y)=\min f$ for all $y\in\states\setminus\states^*$.
Therefore, we infer from Equation~\eqref{eq:def:statestau:upper} that there is some $\lambda>0$ such that
\begin{equation*}
\ut_{2\nicefrac{\tau}{n}}f-\min f\leq\lambda(\ut_{\nicefrac{\tau}{n}}f-\min f),
\end{equation*}
which, because of Equations~\eqref{eq:Tconjugacy} and~\eqref{eq:decomposition}, Proposition~\ref{prop:islowertransoperator} and \ref{L:constantadditivity} implies that
\begin{equation*}
\ut_{\nicefrac{\tau}{n}}^2(f-\min f)
=\ut_{2\nicefrac{\tau}{n}}(f-\min f)
=\ut_{2\nicefrac{\tau}{n}}f-\min f
\leq\lambda(\ut_{\nicefrac{\tau}{n}}f-\min f).
\end{equation*}
Hence, it follows from Proposition~\ref{prop:islowertransoperator}, \ref{L:constantadditivity}, Equations~\eqref{eq:Tconjugacy} and~\eqref{eq:decomposition}, \ref{L:monotone} and~\ref{L:homo} that
\begin{equation}\label{eq:constantinequality:upper}
\ut_\tau f-\min f
=\ut_\tau (f-\min f)
=\ut_{\nicefrac{\tau}{n}}^n(f-\min f)
\leq\lambda^{n-1}(\ut_{\nicefrac{\tau}{n}}f-\min f).
\end{equation}
Consider now any $y\in\states\setminus\states^*$. Since $\ut_{\nicefrac{\tau}{n}}f(y)=\min f$, it follows from Equation~\eqref{eq:constantinequality:upper} that $\ut_{\tau}f(y)\leq\min f$, which in turn implies that $y\notin\states_\tau$. Since this holds for all $y\in\states\setminus\states^*$, we find that $\states_\tau\subseteq\states^*=\states_{\nicefrac{\tau}{n}}$. Furthermore, since $\nicefrac{\tau}{n}\leq\tau$, it follows from Equation~\eqref{eq:increasingfunctionupper} that $\states_{\nicefrac{\tau}{n}}\subseteq\states_\tau$. Hence, we find that $\states_\tau=\states^*$. Since this is true for all $\tau>\tau^*$, it follows from Equation~\eqref{eq:initialconstant:upper} that
\begin{equation*}
\states_\tau=\states^*
\text{ for all $\tau>0$}.
\end{equation*}
Therefore, due to Equation~\eqref{eq:def:statestau:upper}, we find that
\begin{equation*}
\ut_tf(x)>\min f
\asa x\in\states_t
\asa x\in\states_s
\asa\ut_sf(x)>\min f.
\end{equation*}
\end{proof}


\begin{proposition}\label{prop:minmaxproperty}
Let $\lrate$ be a lower transition rate operator. Then for all $f\in\gamblesX$, $x\in\states$ and $t,s>0$:
\begin{align*}
f(x)>\min f\hspace{1.5pt}~\then~\lt_tf(x)>\min f\hspace{1.5pt}&~\asa~\lt_sf(x)>\min f;\\
f(x)<\max f~\then~\lt_tf(x)<\max f&~\asa~\lt_sf(x)<\max f;\\
f(x)>\min f\hspace{1.5pt}~\then~\ut_tf(x)>\min f\hspace{1.5pt}&~\asa~\ut_sf(x)>\min f;\\
f(x)<\max f~\then~\ut_tf(x)<\max f&~\asa~\ut_sf(x)<\max f;
\end{align*}
\end{proposition}
\begin{proof}
The first implication [$f(x)>\min f\then\lt_tf(x)>\min f$] follows from Lemma~\ref{lemma:stayingabovetheminimum} and the first equivalence [$\lt_tf(x)>\min f\asa\lt_sf(x)>\min f$] follows from Lemma~\ref{lemma:constantsignlower}. 
Since $\ut_0f=f$, the third implication [$f(x)>\min f\then\ut_tf(x)$] follows from Lemma~\ref{lemma:increasingupper}. The third equivalence [$\ut_tf(x)>\min f\asa\ut_sf(x)>\min f$] follows from Lemma~\ref{lemma:constantsignupper}.
The rest of the result now follows directly because we know from Equation~\eqref{eq:Tconjugacy} that $\ut_tf(x)=-\lt_t(-f)(x)$, $\ut_sf(x)=-\lt_s(-f)(x)$ and $\max f=-\min(-f)$.
\end{proof}

\begin{corollary}\label{corol:constantsignprob}
Let $\lrate$ be a lower transition rate operator. Then for all $A\subseteq\states$, all $x\in\states$ and all $t, s>0$:
\begin{align*}
x\in A~\then~\lt_t\ind{A}(x)>0&~\asa~\lt_s\ind{A}(x)>0;\\
x\notin A~\then~\lt_t\ind{A}(x)<1&~\asa~\lt_s\ind{A}(x)<1;\\
x\in A~\then~\ut_t\ind{A}(x)>0&~\asa~\ut_s\ind{A}(x)>0;\\
x\notin A~\then~\ut_t\ind{A}(x)<1&~\asa~\ut_s\ind{A}(x)<1.
\end{align*}
\end{corollary}
\begin{proof}
If $A=\emptyset$ or $A=\states$, the result follows trivially from Proposition~\ref{prop:islowertransoperator} and \ref{L:bounds}. In all other cases, the result follows directly from Proposition~\ref{prop:minmaxproperty}, with \mbox{$f=\ind{A}$}.
\end{proof}

\begin{proof}[Proof of Proposition~\ref{prop:regulariffonestep}]
If $t=0$, we know from Proposition \ref{prop:operatordifferential} that $\lt_t=I$, which implies that $\lt_t^n=I=\lt_t$ and $\ut_t^n=I=\ut_t$ for all $n\in\nats$. In that case, Definitions~\ref{def:regularlyabsorbing} and~\ref{def:onestepabsorbing} are trivially equal.
If $t>0$, then since we know from Equations~\eqref{eq:decomposition} and~\eqref{eq:Tconjugacy} that $\lt_t^n=\lt_{nt}$ and $\ut_t^n=\ut_{nt}$ for all $n\in\nats$, the equivalence of Definitions~\ref{def:regularlyabsorbing} and~\ref{def:onestepabsorbing} follows directly from Corollary~\ref{corol:constantsignprob}.
\end{proof}

\begin{proof}[Proof of Corollary~\ref{corol:ergodiciffdiscrete1stepabsorbing}]
This result is a trivial consequence of Corollary~\ref{corol:ergodiciffdiscreteregularlyabsorbing} and Proposition~\ref{prop:regulariffonestep}.
\end{proof}

\begin{proof}[Proof of Proposition~\ref{prop:iffuperreachable}]
First assume that $\ut_t\ind{x}(y)>0$. It then follows from Proposition~\ref{prop:approximation} and Equations~\eqref{eq:Tconjugacy} and~\eqref{eq:Qconjugacy} that there is some $n\in\nats$ such that $n\geq t\norm{\lrate}$ and
\begin{equation}\label{eq:theproductispositive}
\Big((I+\frac{t}{n}\urate)^n\ind{x}\Big)(y)
>0.
\end{equation}
Let $\Delta\coloneqq\nicefrac{t}{n}\geq0$ and define $\lt_*\coloneqq I+\Delta\lrate$. Since $n\geq t\norm{\lrate}$ implies that $\Delta\norm{\lrate}\leq1$, it then follows from Proposition~\ref{prop:fromQtoL} that $\lt_*$ is a lower transition operator.
Therefore, for all $z\in\states$ and $w\in\states$, it follows from~\ref{L:bounds} that $c(w,z)\coloneqq\big(\ut_*\ind{z}\big)(w)\geq0$. For all $z\in\states$, we now have that
\begin{align*}
\ut_*\ind{z}
=\sum_{w\in\states}\ind{w}\cdot\big(\ut_*\ind{z}\big)(w)
=\sum_{w\in\states}c(w,z)\ind{w}.
\end{align*}
Hence, for all $x_n\in\states$, it follows from Equation~\eqref{eq:Tconjugacy}, \ref{L:subadditivity} and~\ref{L:monotone} that
\begin{align*}
\ut_*^n\ind{x_n}
=\ut_*^{n-1}\ut_*\ind{x_n}
&=
\ut_*^{n-1}\sum_{x_{n-1}\in\states}c(x_{n-1},x_n)\ind{x_{n-1}}
\leq
\sum_{x_{n-1}\in\states}c(x_{n-1},x_n)\ut_*^{n-1}\ind{x_{n-1}}
\end{align*}
and, by continuing in this way, that
\begin{align*}
\ut_*^n\ind{x_n}
\leq
\sum_{x_{n-1}\in\states}c(x_{n-1},x_n)
\sum_{x_{n-2}\in\states}c(x_{n-2},x_{n-1})
\cdots
\sum_{x_1\in\states}c(x_1,x_2)\ut_*\ind{x_1}.
\end{align*}
Therefore, for all $x_n\in\states$ and $x_0\in\states$, we find that
\begin{align*}
\big(\ut_*^n\ind{x_n}\big)(x_0)
&\leq
\sum_{x_{n-1}\in\states}c(x_{n-1},x_n)
\sum_{x_{n-2}\in\states}c(x_{n-2},x_{n-1})
\cdots
\sum_{x_1\in\states}c(x_1,x_2)c(x_0,x_1).
\end{align*}
Hence, if we let $x_0\coloneqq y$ and $x_n\coloneqq x$, it follows from Equation~\eqref{eq:theproductispositive} that
\begin{align*}
\sum_{x_{n-1}\in\states}c(x_{n-1},x_n)
\sum_{x_{n-2}\in\states}c(x_{n-2},x_{n-1})
\cdots
\sum_{x_1\in\states}c(x_1,x_2)c(x_0,x_1)
>0.
\end{align*}
This implies that there is some sequence $y=x_0,x_1,\dots,x_n=x$ such that
\begin{align*}
c(x_{n-1},x_n)
c(x_{n-2},x_{n-1})
\cdots
c(x_1,x_2)c(x_0,x_1)
>0.
\end{align*}
Since each of the factors in this product is non-negative, it follows that $c(x_{k-1},x_{k})>0$ for all $k\in\{1,\dots,n\}$. Therefore, for any $k\in\{1,\dots,n\}$ such that $x_k\neq x_{k-1}$, it follows that
\begin{align*}
\urate(\ind{x_k})(x_{k-1})
&=\frac{1}{\Delta}\Big(\ind{x_k}(x_{k-1})+\Delta\urate(\ind{x_k})(x_{k-1})\Big)\\
&=\frac{1}{\Delta}\Big((\ind{x_k}+\Delta\urate(\ind{x_k}))(x_{k-1})\Big)\\
&=\frac{1}{\Delta}\Big(\big((I+\Delta\urate)\ind{x_k}\big)(x_{k-1})\Big)
=\frac{1}{\Delta}\Big(\big(\ut_*\ind{x_k}\big)(x_{k-1})\Big)
=\frac{1}{\Delta}c(x_{k-1},x_k)>0.
\end{align*}
If $x_k\neq x_{k-1}$ for all $k\in\{1,\dots,n\}$, this implies that $x$ is upper reachable from $y$. 
Otherwise, let $x'_0,\dots,x'_m$ be a new sequence, obtained by removing from $x_0,\dots,x_n$ those elements $x_k$ for which $x_k=x_{k-1}$; $n-m$ is the number of elements that is removed. 
Then $x'_0=y$, $x'_m=x$ and, for all $k\in\{1,\dots,m\}$, we have that $x'_k\neq x'_{k-1}$ and $\urate(\ind{x'_k})(x'_{k-1})>0$. Therefore, $x$ is upper reachable from $y$.

Conversely, assume that $x$ is upper reachable from $y$, meaning that there is some sequence $y=x_0,x_1,\dots,x_n=x$ such that, for all $k\in\{1,\dots,n\}$, $x_k\neq x_{k-1}$ and $\urate\ind{x_k}(x_{k-1})>0$. 
If $n=0$, then $x=y$ and therefore, it follows from Corollary~\ref{corol:constantsignprob} that $\ut_t\ind{x}(y)>0$.
Hence, for the remainder of this proof, we may assume that $n\geq 1$.
Fix any $k\in\{1,\dots,n\}$. We then have that $\ut_0 \ind{x_k}(x_{k-1})=\ind{x_k}(x_{k-1})=0$ and
\begin{align*}
\frac{d}{ds}\ut_s \ind{x_k}(x_{k-1})\Big\vert_{s=0}
&=-\frac{d}{ds}\lt_s(-\ind{x_k})(x_{k-1})\Big\vert_{s=0}\\
&=-\lrate(\lt_0(-\ind{x_k}))(x_{k-1})
=-\lrate(-\ind{x_k})(x_{k-1})
=\urate(\ind{x_k})(x_{k-1})>0,
\end{align*}
where the first equality follows from Equation~\eqref{eq:Tconjugacy}, the second equality follows from Equation~\eqref{eq:differential} and the last equality follows from Equation~\eqref{eq:Qconjugacy}.
Therefore, there is some $\epsilon_k>0$ such that $\ut_{\epsilon_k}\ind{x_k}(x_{k-1})>0$. Consequently, if we let $c_k\coloneqq\ut_{\epsilon_k}\ind{x_k}(x_{k-1})>0$, then because it follows from Proposition~\ref{prop:islowertransoperator} and~\ref{L:bounds} that $\ut_{\epsilon_k}\ind{x_k}\geq0$, we have that $\ut_{\epsilon_k}\ind{x_k}\geq c_k\ind{x_{k-1}}$.
Let $\epsilon\coloneqq\sum_{k=1}^{n}\epsilon_k>0$. Then
\begin{equation*}
\ut_\epsilon\ind{x_n}=\ut_{\epsilon_1}\cdots\ut_{\epsilon_{n-1}}\ut_{\epsilon_n}\ind{x_n}
\geq c_n\ut_{\epsilon_1}\cdots\ut_{\epsilon_{n-1}}\ind{x_{n-1}}
\geq \ldots
\geq \left(\prod_{k=1}^n c_k\right) \ind{x_0},
\end{equation*}
where the equality follows from Equations~\eqref{eq:Tconjugacy} and~\eqref{eq:decomposition} and where the inequalities follow from Proposition~\ref{prop:islowertransoperator}, \ref{L:monotone}, \ref{L:homo} and Equation~\eqref{eq:Tconjugacy}. Therefore, we find that
\begin{equation*}
\ut_\epsilon\ind{x}(y)=\ut_\epsilon\ind{x_n}(x_0)
\geq \prod_{k=1}^n c_k \ind{x_0}(x_0)
=\prod_{k=1}^n c_k>0,
\end{equation*}
which implies that $\ut_t\ind{x}(y)>0$ because of Corollary~\ref{corol:constantsignprob}.
\end{proof}

\begin{proof}[Proof of Proposition~\ref{prop:ifflowerreachable}]
Let $\{A_k\}_{k\in\natswith}$ and $n$ be defined as in Definition~\ref{def:lowerreachability}. We need to prove that $x\in A_n$ if and only if $\lt_t\ind{A}(x)>0$.

First assume that $\lt_t\ind{A}(x)>0$. It then follows from Proposition~\ref{prop:approximation} that there is some $m\in\nats$ such that $m\geq n$, $m\geq t\norm{\lrate}$ and
\begin{equation}\label{eq:theproductispositive:lower}
\Big((I+\frac{t}{m}\lrate)^m\ind{A}\Big)(x)
>0.
\end{equation}
Let $\Delta\coloneqq\nicefrac{t}{m}\geq0$ and define $\lt_*\coloneqq I+\Delta\lrate$. Since $m\geq t\norm{\lrate}$ implies that $\Delta\norm{\lrate}\leq1$, it then follows from Proposition~\ref{prop:fromQtoL} that $\lt_*$ is a lower transition operator.
Consider any $k\in\natswith$ such that $k\leq m$ and any $y\in\states\setminus A_{k+1}$. Since $A_k\subseteq A_{k+1}$, this implies that $y\notin A_k$. Assume \emph{ex absurdo} that $\lrate(\ind{A_k})(y)>0$. It then follows from Equation~\eqref{eq:lowerreachability} that $y\in A_{k+1}$, a contradiction. Hence, we find that $\lrate(\ind{A_k})(y)\leq0$. Since $y\notin A_k$, it follows from \ref{LR:subadditive} and~\ref{LR:nondiagpositive} that $\lrate(\ind{A_k})(y)\geq\sum_{z\in A_k}\lrate(\ind{z})(y)\geq0$. Hence, we infer that $\lrate(\ind{A_k})(y)=0$. Furthermore, since $y\notin A_k$, we als have that $\ind{A_k}(y)=0$. Hence, we find that $(\lt_*\ind{A_k})(y)=((I+\Delta\lrate)\ind{A_k})(y)=0$. Since this holds for all $y\in\states\setminus A_{k+1}$, there is some $c_k>0$ such that $\lt_*\ind{A_k}\leq c_k\ind{A_{k+1}}$. Due to \ref{L:homo} and \ref{L:monotone}, this implies that $\lt_*^{m-k}\ind{A_k}\leq c_k\lt_*^{m-k-1}\ind{A_{k+1}}$. Since this holds for all $k\in\natswith$ such that $k\leq m$, we find that
\begin{equation*}
\lt_*^m\ind{A_0}
\leq c_0\lt_*^{m-1}\ind{A_1}
\leq c_0c_1\lt_*^{m-2}\ind{A_2}
\leq\dots
\leq c_0c_1\dots c_{m-1}\ind{A_m}.
\end{equation*}
Therefore, since $A_0=A$, it follows from Equation~\eqref{eq:theproductispositive:lower} that $x\in A_m$. Since $A_n=A_{n+1}$, it follows from Equation~\eqref{eq:lowerreachability} that $A_r=A_n$ for all $r\geq n$ and therefore, in particular, that $A_m=A_n$. Since $x\in A_m$, this implies that $x\in A_n$.

Conversely, assume that $x\in A_n$.
If $n=0$, then $A_n=A_0=A$ and therefore $x\in A$, which, due to Corollary~\ref{corol:constantsignprob}, implies that $\lt_t\ind{A}(x)>0$. Therefore, for the remainder of this proof, we may assume that $n\geq 1$.
Fix any $k\in\{0,\dots,n-1\}$. Consider any $y\in A_{k+1}\setminus A_k$. Then $\lt_0\ind{A_k}(y)=\ind{A_k}(y)=0$ and
\begin{equation*}
\frac{d}{ds}\lt_s \ind{A_k}(y)\Big\vert_{s=0}
=\lrate(\lt_0(\ind{A_k}))(y)
=\lrate(\ind{A_k})(y)>0,
\end{equation*}
where the first equality follows from Equation~\eqref{eq:differential} and the inequality follows from Equation~\eqref{eq:lowerreachability}. Therefore, there is some $\epsilon_{k,y}>0$ such that $\lt_s\ind{A_k}(y)>0$ for all $s\in(0,\epsilon_{k,y}]$. Hence, if we let $\epsilon_k\coloneqq\min_{y\in A_{k+1}\setminus A_k}\epsilon_{k,y}$, then $\lt_{\epsilon_k}\ind{A_k}(y)>0$ for all $y\in A_{k+1}\setminus A_k$. For all $y\in A_k$, it follows from Corollary~\ref{corol:constantsignprob} that $\lt_{\epsilon_k}\ind{A_k}(y)>0$. Hence, in summary, we have that $\lt_{\epsilon_k}\ind{A_k}(y)>0$ for all $y\in A_{k+1}$. Since we know from Proposition~\ref{prop:islowertransoperator} and \ref{L:bound} that $\lt_{\epsilon_k}\ind{A_k}\geq0$, this implies that there is some $c_k>0$ such that $\lt_{\epsilon_k}\ind{A_k}\geq c_k\ind{A_{k+1}}$. Let $\epsilon\coloneqq\sum_{k=0}^{n-1}\epsilon_k$. Then
\begin{equation*}
\lt_\epsilon\ind{A_0}=\lt_{\epsilon_{n-1}}\cdots\lt_{\epsilon_{1}}\lt_{\epsilon_0}\ind{A_0}
\geq c_0\lt_{\epsilon_{n-1}}\cdots\lt_{\epsilon_{1}}\ind{A_1}
\geq \ldots
\geq \left(\prod_{k=0}^{n-1} c_k\right) \ind{A_{n}},
\end{equation*}
where the equality follows from Equation~\eqref{eq:decomposition} and the inequalities follow from Proposition~\ref{prop:islowertransoperator}, \ref{L:homo} and~\ref{L:monotone}. Therefore, we find that
\begin{equation*}
\lt_\epsilon\ind{A}(x)=\lt_\epsilon\ind{A_0}(x)
\geq \left(\prod_{k=0}^{n-1} c_k\right) \ind{A_n}(x)
=\prod_{k=0}^{n-1}c_k>0,
\end{equation*}
which implies that $\lt_t\ind{A}(x)>0$ because of Corollary~\ref{corol:constantsignprob}.
\end{proof}

\begin{proof}[Proof of Theorem~\ref{theo:finalequivalence}]
Fix any $t>0$. It then follows from Corollary~\ref{corol:ergodiciffdiscrete1stepabsorbing} that $\lrate$ is ergodic if and only if $\lt_t$ is $1$-step absorbing or, equivalently, because of Definition~\ref{def:onestepabsorbing}, if
\begin{equation*}
\mathcal{X}_{\mathrm{1A}}\coloneqq
\{
x\in\states\colon
\min\ut_t\ind{x}>0
\}\neq\emptyset
\text{~~and~~}
(\forall x\in\states\setminus\mathcal{X}_{\mathrm{1A}})~
\lt_t\ind{\mathcal{X}_{\mathrm{1A}}}(x)>0.
\end{equation*}
The result now follows immediately because we know from Proposition~\ref{prop:iffuperreachable} that, for all $x\in\states$,
\begin{equation*}
\min\ut_t\ind{x}>0
~\Leftrightarrow~
(\forall y\in\states)~\ut_t\ind{x}(y)>0
~\Leftrightarrow~
(\forall y\in\states)~y\dotabovearrow x,
\end{equation*}
and because we know from Proposition~\ref{prop:ifflowerreachable} that, for all $x\in\states\setminus\mathcal{X}_{\mathrm{1A}}$, $\lt_t\ind{\mathcal{X}_{\mathrm{1A}}}(x)>0$ if and only if $x\dotbelowarrow \mathcal{X}_{\mathrm{1A}}$.
\end{proof}

\end{document}